\documentclass[letterpaper, 11pt,  reqno]{amsart}

\usepackage[margin=1.2in,marginparwidth=1.5cm, marginparsep=0.5cm]{geometry}

\usepackage{amsmath, mathrsfs, amssymb,amscd,amsthm,amsxtra, esint}
\usepackage{tikz} 
\usepackage{color}

\usepackage[implicit=true]{hyperref}

\usepackage{cases}

\usepackage{upgreek}

\allowdisplaybreaks[2]

\definecolor{gr}{rgb}   {0.,   0.69,   0.23 }
\definecolor{bl}{rgb}   {0.,   0.5,   1. }
\definecolor{mg}{rgb}   {0.85,  0.,    0.85}
\definecolor{yl}{rgb}   {0.8,  0.7,   0.}
\definecolor{or}{rgb}  {0.7,0.2,0.2}

\usetikzlibrary{shapes.misc}
\usetikzlibrary{shapes.symbols}
\usetikzlibrary{shapes.geometric}

\tikzset{
	ddot/.style={circle,fill=white,draw=black,inner sep=0pt,minimum size=0.8mm},
	>=stealth,
	}

\tikzset{
	ddot2/.style={circle,fill=black,draw=black,inner sep=0pt,minimum size=0.8mm},
	>=stealth,
	}

\newtheorem{theorem}{Theorem} [section]

\newtheorem{lemma}[theorem]{Lemma}

\newtheorem{remark}[theorem]{Remark}

\newtheorem{definition}[theorem]{Definition}
\newtheorem{corollary}[theorem]{Corollary}

\newtheorem{oldtheorem}{Theorem}

\DeclareMathOperator{\Id}{Id}
\DeclareMathOperator{\sgn}{sgn}

\DeclareMathOperator{\Ker}{Ker}

\newcommand{\1}{\hspace{0.5mm}\text{I}\hspace{0.5mm}}
\newcommand{\II}{\text{I \hspace{-2.8mm} I} }
\newcommand{\I}{\mathcal{I}}

\newcommand{\noi}{\noindent}
\newcommand{\Z}{\mathbb{Z}}
\newcommand{\R}{\mathbb{R}}

\newcommand{\T}{\mathbb{T}}
\newcommand{\bul}{\bullet}

\renewcommand{\H}{\mathcal{H}}

\newcommand{\CC}{\mathcal{C}}

\renewcommand{\L}{\mathcal{L}}

\newcommand{\RR}{\mathcal{R}}

\newcommand{\F}{\mathcal{F}}

\newcommand{\al}{\alpha}
\newcommand{\be}{\beta}
\newcommand{\dl}{\delta}
\newcommand{\updl}{\updelta}

\newcommand{\Dl}{\Delta}
\newcommand{\eps}{\varepsilon}

\newcommand{\g}{\gamma}
\newcommand{\G}{\Gamma}
\newcommand{\ld}{\lambda}
\newcommand{\Ld}{\Lambda}
\newcommand{\s}{\sigma}

\newcommand{\ft}{\widehat}

\newcommand{\cj}{\overline}
\newcommand{\dx}{\partial_x}
\newcommand{\dt}{\partial_t}
\newcommand{\dd}{\partial}

\newcommand{\ta}{\theta}

\renewcommand{\l}{\ell}

\newcommand{\Gdl}{\mathcal{G}_{\dl} }

\newcommand{\les}{\lesssim}
\newcommand{\ges}{\gtrsim}

\newcommand{\jb}[1]
{\langle #1 \rangle}

\newcommand{\ind}{\mathbf 1}

\newcommand{\M}{\mathcal{M}}

\newcommand{\N}{\mathbb{N}}

\newcommand{\NN}{\mathcal{N}}
\newcommand{\cL}{\mathcal{L}}
\newcommand{\cC}{\mathcal{C}}

\newcommand{\cD}{\mathscr{D}}

\newcommand{\W}{\mathcal{W}}

\newcommand{\Lip}{\mathrm{Lip}}
\newcommand{\uw}{U^w}
\newcommand{\uu}{\mathbf{u}}

\newcommand{\z}{\zeta}

\newcommand{\Ta}{\Theta}

\newcommand{\sub}{\substack}

\newcommand{\BO}{\text{\rm BO} }
\newcommand{\KDV}{\text{\rm KdV} }

\newcommand{\NLS}{\text{\rm NLS} }
\newcommand{\dNLS}{\text{\rm dNLS} }

\newcommand{\too}{\longrightarrow}

\newtheorem*{ackno}{Acknowledgements}

\numberwithin{equation}{section}
\numberwithin{theorem}{section}

\makeatletter
\@namedef{subjclassname@2020}{\textup{2020} Mathematics Subject Classification}
\makeatother

\begin{document}
\baselineskip = 14pt

\title[Nonlinear PDEs with modulated dispersion IV]
{Nonlinear PDEs with modulated dispersion IV:\\
normal form approach and unconditional uniqueness}

\author[M.~Gubinelli, G.~Li, J.~Li, and T.~Oh]
{Massimiliano Gubinelli, Guopeng Li, Jiawei Li,  and Tadahiro Oh}

\address{Massimiliano Gubinelli\\
Mathematical Institute\\ University of Oxford\\ United Kingdom}

\email{gubinelli@maths.ox.ac.uk}

\address{Guopeng Li, 
School of Mathematics and Statistics, Beijing Institute of Technology, Beijing 100081, China}

\email{guopeng.li@bit.edu.cn}

\address{Jiawei Li, School of Mathematics\\
The University of Edinburgh\\
and The Maxwell Institute for the Mathematical Sciences\\
James Clerk Maxwell Building\\
The King's Buildings\\
Peter Guthrie Tait Road\\
Edinburgh\\ 
EH9 3FD\\
 United Kingdom}

\email{jiawei.li@ed.ac.uk}

\address{
Tadahiro Oh, 
School of Mathematics\\
The University of Edinburgh\\
and The Maxwell Institute for the Mathematical Sciences\\
James Clerk Maxwell Building\\
The King's Buildings\\
Peter Guthrie Tait Road\\
Edinburgh\\
EH9 3FD\\
 United Kingdom\\
and School of Mathematics and Statistics, Beijing Institute of Technology, Beijing 100081, China}

%

\email{hiro.oh@ed.ac.uk}



\begin{abstract}

We study the modulated  Korteweg-de~Vries equation (KdV)
on the circle
with a time non-homogeneous modulation acting on the linear dispersion term.
By adapting the normal form approach to the modulated setting, 
we prove sharp unconditional uniqueness
of solutions to the modulated KdV in $L^2(\T)$
if a modulation is  sufficiently irregular.
For example, this result implies that if the modulation is given by 
a sample path of a fractional Brownian motion 
with Hurst index $0 < H < \frac 25$, the modulated KdV on the circle is 
unconditionally well-posed in $L^2(\T)$.
At a philosophical level, 
our normal form approach can be viewed
as 
a {\it controlled path approach
to nonlinear Young integration}, 
which 
allows for 
 the construction of solutions
to the modulated KdV
(and the associated nonlinear Young integral)
{\it without} assuming any positive (H\"older) regularity in time.
As an interesting byproduct of our normal form approach, 
we  obtain an improved Euler approximation scheme 
as compared to the classical sewing lemma approach.

We also establish analogous sharp unconditional uniqueness results
for the modulated Benjamin-Ono equation
and the modulated derivative nonlinear Schr\"odinger equation (NLS)
with a quadratic nonlinearity.
In the appendix, 
we 
prove sharp unconditional uniqueness 
of the cubic modulated NLS on the circle in $H^{\frac 16}(\T)$.

\end{abstract}

%
\maketitle

\tableofcontents

\section{Introduction}\label{SEC:1}

\subsection{Modulated dispersive equations}
\label{SUBSEC:1.1}

In the works \cite{CG1, CGLLO1}
 with K.\,Chouk, we  
 considered a {\it modulated} dispersive equation of the following form:
 \begin{align}
\begin{cases}
\dt u +  L u \cdot \dt w=  \NN(u)  \\ 
u|_{t = 0} = u_0,
\end{cases}
\qquad ( t, x) \in \R_+ \times \M,
\label{ME1}
\end{align}

\noi
where $\M  = \R$ or $\T = \R/ (2\pi \Z)$,\footnote{By convention, we endow
$\T$ with the normalized Lebesgue measure $ dx_{_\T} =  (2\pi)^{-1}dx$
such that we do not need to carry factors involving $2\pi$.}
$\NN(u)$ denotes the nonlinearity, 
and 
 $w:\R_+\to\R$ is a continuous function of time, 
called a {\it modulation}, 
 acting on the linear dispersion term $Lu$.
For a specific modulation such as that given by a Brownian motion, 
stochastic calculus can be used to study 
the equation \eqref{ME1};  see \cite{DD, DT, DR2, Ste}.
For a general modulation $w$, however, 
such an approach based on stochastic calculus
is not available.
For example, 
consider 
the modulated 
 Korteweg-de Vries equation (KdV):
\begin{equation}
\label{kdv1}
\dt u+  \dx^3 u \cdot \dt w =\dx u^2
\end{equation}

\noi
which 
naturally appears 
as a model 
for
weakly nonlinear long waves in an inhomogeneous waveguide
(\cite{CMG, HZ}).
Here, 
the modulation $w$ is taken to be periodic but not differentiable
and thus 
a stochastic analysis approach does not apply to this model.

In \cite{CG1, CGLLO1},  
 with K.\,Chouk, we  
 developed 
  a {\it pathwise} approach 
to study the modulated dispersive equation~\eqref{ME1}
by exploiting {\it irregularity} of the modulation function $w$
(see Definition \ref{DEF:ir})
and  the temporal regularity of a solution~$u$.
In particular, 
in \cite{CGLLO1}, 
we established 
{\it regularization-by-noise}
phenomena for the   modulated 
KdV \eqref{kdv1}
 on both the circle and the real line.
For example, 
we proved that, given {\it any} $s \in \mathbb R$, the modulated KdV \eqref{kdv1} on $\T$ with a sufficiently irregular modulation is locally well-posed in $H^s(\T)$. 
On the one hand, since the 70's, 
there have been intensive research activities
on regularization by noise 
for stochastic differential equations
and stochastic parabolic partial differential equations;
see
 \cite{
GP, 
KR, 
 FGP, GG0, G22, RT, CGLLO1,
 ABLL} for the  references therein.
See also   survey works \cite{Flan, Gess}.
On the other hand, 
prior to our work~\cite{CGLLO1}, 
results on regularization by noise 
for  dispersive equations
were primarily limited to random initial data\,/\,additive noises of super-critical regularity;
see, for example, 
\cite{BO96, BT08, CO, BT3, BOP2, Poc, OP, 
GKO2, DNY2, DNY3, OOT1, Bring, OOT2, BDNY}.
In~\cite{CGLLO1}, 
we also established 
other forms of regularization by noise
such as 
semilinearization and nonlinear smoothing
(of arbitrary order)
for the modulated KdV \eqref{kdv1}
and other modulated dispersive equations
such as the modulated Benjamin-Ono equation (BO):
\begin{equation}
\label{BO}
\dt u-    \H\dx^2 u \cdot \dt w=\dx u^2,
\end{equation}

\noi
where $\H$ denotes the
Hilbert transform
with the Fourier multiplier $- \ind_{n\ne 0}\cdot i \sgn (n)$.
See~\cite{CGLLO1}
for a further discussion on various examples.
We also mention 
recent works
\cite{Tanaka, Robert1, Robert2, Robert3, Robert4, GGLO}
on pathwise well-posedness of  various modulated dispersive equations;
see \cite[Remark~1.14]{CGLLO1}.

In a recent work \cite{CGLLO2} with A.~Chapouto, we studied
pathwise well-posedness of stochastic modulated dispersive equations
with multiplicative noises
and established a new 
regularization-by-noise phenomenon by exploiting the nonlinear interaction between the unknown and the noise.
For example, 
for 
the stochastic KdV with 
a multiplicative fractional-in-time noise in the Young regime, 
we showed that 
irregularity of the modulation induces
smoothing on the stochastic convolution
in a pathwise manner, 
where 
a gain of spatial regularity  becomes (arbitrarily) larger for more irregular modulations.

Since $w$ in \eqref{ME1}
is assumed to be  only continuous, 
the modulated dispersive equation~\eqref{ME1} 
is  formal at this point.
To bypass this problem, 
 we studied 
 the following 
 Duhamel formulation (= mild formulation) of \eqref{ME1}
 in \cite{CG1, CGLLO1}:
\begin{equation}
u(t) = U^w(t) u_0 +  U^w({t}) \int_0^t  U^w({t'})^{-1}\NN( u(t') ) dt', 
\label{mild1}
\end{equation}

\noi
where $U^w(t) = e^{- w(t) L }$ denotes the modulated linear propagator
and we impose $w(0) = 0$ such that  
 $U^w(0) = \Id$.\footnote{The normalization $w(0) = 0$
 is not an additional restriction since only the time derivative $\dt w$
 appears in the modulated equation \eqref{ME1}.}
Define
 the 
 modulated interaction representation
$\uu$ of the unknown $u$ by setting
\begin{align}
\uu(t)=\uw(t)^{-1}u(t), 
\label{int1} 
\end{align}

\noi
for which 
the Duhamel formulation \eqref{mild1}
becomes 
\begin{equation}
\uu(t) = u_0 + \int_0 ^t \uw(t')^{-1}\NN(\uw(t') \uu(t'))dt'.
\label{mild2}
\end{equation}

\noi
Our basic strategy 
in \cite{CGLLO1, CGLLO2}
was 
to first make sense of the integral term in~\eqref{mild2}
as a  nonlinear Young integral
$\I^X(\uu)$ with a given driver $X$
(see, for example, \eqref{K1}
for the modulated KdV)
by applying  the sewing lemma  (Lemma \ref{LEM:sew})
from 
the theory of controlled rough paths~\cite{Gub04}
developed by the first author.
We then studied 
the resulting nonlinear Young differential equation:
\begin{equation}
\uu = u_0 + \I^X(\uu)
\label{YDE0}
\end{equation}

\noi
(possibly with an additive\,/\,multiplicative stochastic forcing)
by a contraction argument. 
We point out that 
 the (sufficient) temporal regularity of
the modulated interaction representation~$\uu$
was essential 
in applying the sewing lemma 
to construct
a nonlinear Young integral.

\medskip

Now, let us turn to the uniqueness property of  solutions.
Recall 
from  \cite{KATO} 
that  a Cauchy problem is  said to be {\it unconditionally (locally) well-posed}
in $H^s(\M)$
if for every initial condition $u_0 \in H^s(\M)$,
there exist $\tau >0$ and a unique solution $u \in C([0, \tau ];H^s(\M))$
with  $u|_{t = 0} = u_0$.
We refer to such uniqueness in the entire class $C([0, \tau ];H^s(\M))$, 
without intersecting with any auxiliary function space, as {\it unconditional uniqueness}.
Unconditional uniqueness is a concept of uniqueness which does not depend
on how solutions are constructed
and thus is of fundamental importance.

Given a driver $X$ of temporal regularity $\frac 12 < \g < 1$, 
a solution $\uu$ to \eqref{YDE0} constructed in \cite{CG1, CGLLO1}
belongs to $\CC^\al([0, \tau] ; H^s(\M))$
for some $0 < \al < 1$ with $\g + \al > 1$
such that the nonlinear Young integral $\I^X(\uu)$
makes sense via the sewing lemma.
As such, 
 the uniqueness, at the level of the modulated interaction representation,  holds 
 {\it conditionally}, i.e.~only  in the class 
 $\CC^\al([0, \tau]; H^s(\M))$;
 see Remark \ref{REM:control}.\footnote{\label{FT:3}Note that the uniqueness
 of solutions to modulated dispersive equations
 constructed 
in \cite{Robert2, Robert3, Robert4}
also holds only conditionally, 
namely only in 
 the  (refined) Fourier restriction norm space $X^s_\tau $.}

Our main goal in this paper is to present
an alternative solution theory based
on the {\it normal form method}, 
where we do {\it not} require any positive  temporal regularity on
the modulated interaction representation $\uu$.
In particular, we construct a solution 
by working  in $C([0, \tau]; H^s(\M))$
{\it without} utilizing any auxiliary function space.
This approach allows us to 
establish 
the first unconditional uniqueness results
for  modulated dispersive equations;
see Theorems \ref{THM:UU1}, \ref{THM:UU2},  \ref{THM:UU3}, 
and \ref{THM:UU4}.
We point out that 
these unconditional uniqueness results
are {\it sharp} in terms of the spatial regularity;
see Remarks \ref{REM:sharp1} and \ref{REM:sharp2}.

From the viewpoint of nonlinear Young integration theory, 
our approach 
presents a completely new construction 
of the nonlinear Young integral $\I^X(\uu)$
in \eqref{YDE0}
 without using any positive temporal regularity, 
but {\it by making use of 
the structure \eqref{YDE0}
of the unknown $\uu$}
(just like a controlled rough path).
In this sense, the novelty of 
our normal form approach
can be viewed,  
at a philosophical level, 
as 
a controlled path approach
to nonlinear Young integration
which 
significantly improves any existing theory
in terms of the temporal regularity.
See Subsection~\ref{SUBSEC:comp}
for a further discussion.
We also point out that our
normal form approach
provides an improvement on  
 the time discretization scheme
as compared to  the sewing lemma approach 
discussed in  \cite[Subsection~2.3]{CG1}
in terms of both the convergence rate and 
the required temporal regularity of a solution;
see
Remark~\ref{REM:num}.
See also Remark \ref{REM:NFx}.
Lastly, while the main body of the paper is devoted
to the modulated KdV-type 
equations with derivative quadratic nonlinearities, 
we also consider 
the modulated cubic nonlinear Schr\"odinger equation (NLS)
on the circle
in Appendix \ref{SEC:A}.

\subsection{Irregularity of the modulation}

We recall a particular notion of irregularity of the modulation $w$, 
introduced in~\cite{CatellierGubinelli, CG1},
which allowed us 
 to establish regularization by noise  in various  ODE and PDE contexts
in \cite{CatellierGubinelli,CG1, CGLLO1, CGLLO2}.
For further study on the notion of irregularity, 
see  \cite{GG, RT}.

\begin{definition}
\label{DEF:ir}
\rm

Let $\rho>0$ and  $0 < \g < 1$.
Given $T > 0$, we say that a function $w\in C([0,T];\R)$ is $(\rho,\g)$-irregular 
on the time interval $[0, T]$  if we have
\begin{align}
\|\Phi^w\|_{  \W^{\rho,\g}_T} 
:= \sup_{a\in \R} \sup_{0\leq r < t\leq T} \langle a \rangle^\rho \frac{|\Phi^w_{t,r}(a)|}{|t-r|^\g} 
< \infty, 
\label{rho1}
\end{align}

\noi
where 
\begin{align}
\Phi^w_{t,r}(a)
=\int_r^t e^{i  a w(t') } d t'.
\label{rho2}
\end{align}

\noi
We say that $w$ is 
$(\rho,\g)$-irregular 
on $\R_+$  if it is 
$(\rho,\g)$-irregular 
on $[0, T]$ for each finite $T > 0$.

\end{definition}

In~\cite{CatellierGubinelli}, 
the authors proved that a fractional Brownian motion is a $(\rho,\g)$-irregular function.

\begin{oldtheorem}
\label{THM:A}

Let $\{W_t\}_{t\in \R_+}$ be a fractional Brownian motion 
of Hurst index $H\in(0,1)$.
Then, 
for any $\rho < \frac{1}{2H}$,  
there exists $\frac 12 < \g < 1$
such that,  with probability one,  the sample paths of 
$W$ are $(\rho,\g)$-irregular on $\R_+$.

\end{oldtheorem}

Theorem \ref{THM:A}
shows that there exist continuous paths which are $(\rho, \g)$-irregular for arbitrarily large $\rho$. 
In a later work~\cite{GG}, it was shown  that $(\rho,\g)$-irregularity is a generic property of H\"older functions of sufficiently low regularity; see \cite[p.\,2418 and Theorem 3.1]{GG}.

\begin{oldtheorem}
\label{THM:B}
Let  $d \ge 1$.
Given  any $0 < \delta < 1$, 
a generic $\dl$-H\"older continuous  function $w \in C^\delta([0,1];\R^d)$ 
 is $(\rho,\g)$-irregular for any $\rho < \frac1{2\delta}$ with  some $\gamma = \gamma(\rho) \in (\frac 12,1)$.
\end{oldtheorem}

Here, ``generic'' is to be understood according to the notion of {\it prevalence}; see~\cite{GG} for details and the  references therein.

Theorems \ref{THM:A} and \ref{THM:B}
do not only provide important examples of $(\rho, \g)$-irregular paths, 
but also imply that 
 our main results
apply 
to the case
where a modulation is given by a fractional Brownian motion
or a generic  $\dl$-H\"older continuous function.
See, for example, \cite[Corollary~1.5]{CGLLO1}.

\subsection{Nonlinear Young differential equations}
\label{SUBSEC:YDE}

In this subsection, 
we briefly go over the nonlinear Young integration approach
in \cite{CGLLO1}
by taking the modulated 
 KdV~\eqref{kdv1} on~$\T$
 as an example.
See also \cite{HK, G23}
for the construction of nonlinear Young integrals in a more general setting.
For simplicity of the presentation, 
we impose
the mean-zero assumption\footnote{For the modulated KdV \eqref{kdv1}, 
if initial data has non-zero mean $\al_0$, then the following Galilean transformation:
$u(t, x)\mapsto u (t, x + 2\al_0t)-\al_0$, 
along with the conservation of the spatial mean (which can be verified
by arguing as in \cite[Section 6]{CGLLO1}), 
transforms
the equation with a non-zero mean into the mean-zero version.
The same comment applies
to the modulated BO \eqref{BO} and the modulated derivative nonlinear
Schr\"odinger equation \eqref{dNLS1}
studied in this paper.
}
 on solutions to
the modulated KdV \eqref{kdv1}
on $\T$.
We also apply the same convention
to  the other modulated dispersive equations 
considered in this paper.\footnote{Except for the cubic modulated nonlinear Schr\"odinger equation
\eqref{xNLS1}
studied in Appendix \ref{SEC:A}.}

Given  $\rho > 0 $ and  $\frac 12  < \g < 1$, 
fix a 
 $(\rho,\g)$-irregular function $w$ on an interval $[0, T]$.
In terms of 
the modulated interaction representation $\uu$
defined in~\eqref{int1}, 
the modulated KdV~\eqref{kdv1}
is written as 
\begin{equation}
\dt \uu = \uw(t)^{-1} \dx\big( (\uw(t) \uu)^2\big), 
\label{kdv1a}
\end{equation}

\noi
where 
$\uw (t)=e^{-   w(t)\dx^3}  $
denotes the modulated linear propagator for \eqref{kdv1}.
By writing~\eqref{kdv1a} (with initial data $u_0$) in the integral form, we have
\begin{equation}
\uu(t) = u_0 + \int_0 ^t \uw(t')^{-1}
\dx\big( (\uw(t') \uu)^2\big)dt'.
\label{mild3x}
\end{equation}

\noi
As mentioned in Subsection \ref{SUBSEC:1.1}, the main strategy
in \cite{CGLLO1}
was to construct the integral term 
in~\eqref{mild3x}
as a nonlinear Young integral via the sewing lemma, 
and solve the resulting nonlinear Young differential equation.

Let $X^\KDV$ be 
the 
bilinear driver  associated with the modulated KdV, 
defined by 
 \begin{align}
X^\KDV_{t,r}(f_1,f_2)
=\int_r^t  \uw(t')^{-1}  \partial_x
\big( (\uw(t')  f_1)( \uw(t') f_2 )\big) dt'
\label{K1}
\end{align}

\noi
for $0 \le r < t \le T$, 
where $f_1$ and $f_2$ are  functions  on $\T$.
By taking the Fourier transform, we have
\begin{align}
\F_x\big(X^\KDV_{t,r} (f_1,f_2)\big) (n)
= in \sum_{ \substack{n_1, n_2 \in \Z_*\\n = n_1+n_2}}
 \Phi^w_{t,r}(\Xi_\KDV (\bar n))
\ft f_1(n_1)  \ft f_2(n_2),
\label{K2}
\end{align}

\noi
where $\Z_* = \Z \setminus \{0\}$, $\Phi^w_{t,r}$ is as in \eqref{rho2}, 
and $\Xi_\KDV (\bar n)$ denotes the {\it resonance function}\footnote{Here, we follow
the terminology in \cite{Tao2}.  
We point out that $\Xi_\KDV (\bar n)$ is also called the modulation function (see \cite{KOY})
but we do not use this latter terminology to avoid
confusion with a modulation function $w$.
}
 for KdV given by 
\begin{align}
\begin{split}
\Xi_\KDV (\bar n) &  = \Xi_\KDV (n,n_1,n_2) 
= - n^3+ n_1^3 +n_2^3\\
&  = - 3n n_1 n_2, 
\end{split}
\label{K3a}
\end{align}

\noi
where the last equality holds under  $n = n_1+ n_2$.
As pointed out in \cite[Section 3]{CGLLO1}, 
local well-posedness of the modulated KdV \eqref{kdv1}
then follows 
once we show that 
the bilinear driver $X^\KDV_{t, r}$ is 
bounded from $H^s(\T) \times
H^s(\T)$ into $H^s(\T)$
with the norm $\sim |t- r|^\g$, 
uniformly in  $0 \le r < t \le T$
(recall that $\g > \frac 12$), 
which in turn follows from 
an elementary computation, 
using \eqref{rho1}, \eqref{K3a}, 
and 
 Lemma \ref{LEM:SUM};
 see
 \cite[Proposition~4.1]{CGLLO1}.
See also Lemma \ref{LEM:nonlin1}.

As a result, we obtained
the following local well-posedness result
for  the modulated KdV~\eqref{kdv1} on $\T$
(\cite[Theorem 1.3\,(i)]{CGLLO1}). 
We also recall the corresponding 
local well-posedness result
for the modulated BO \eqref{BO} on $\T$
(\cite[Theorem 1.11\,(i)]{CGLLO1}).

\begin{oldtheorem}\label{THM:1}
Given $\rho \ge\frac 12$,  $\frac12< \g < 1$, and $T> 0$, 
let  $w$ be $(\rho,\g)$-irregular on $[0, T]$ in the sense of Definition~\ref{DEF:ir}.

\smallskip

\noi
\textup{(i)}
Suppose that $\rho \ge \frac 12$ and $s\in \R$
satisfy one of the following conditions\textup{:}
\begin{align}
\begin{split}
\textup{(i.a)} &\ \  \tfrac 12 \le \rho \le \tfrac 34 \quad \text{and} \quad 
s > \tfrac 32 - 3 \rho, \\
\textup{(i.b)} &\ \  \rho > \tfrac 34\quad  \text{and} \quad  s \ge - \rho.
\end{split}
\label{nonlin1b}
\end{align}

\noi
Then, the modulated KdV equation \eqref{kdv1}
on  $\T$
is 
locally well-posed in $H^s(\T)$
such that the solution $u$ belongs to the class
$ \cD_{w}^s([0, \tau] \times \T)$
defined in Remark \ref{REM:control}\,(i) below, 
where $\tau \in (0, T]$
denotes the local existence time.

\smallskip

\noi
\textup{(ii)}
Suppose that $\rho \ge 1$ and $s\in \R$
satisfy one of the following conditions\textup{:}
\begin{align}
\begin{split}
\textup{(ii.a)} &\ \   \rho =  1 \quad \text{and} \quad s > - \tfrac 12, \\
\textup{(ii.b)} &\ \  \rho >  1 \quad  \text{and} \quad  s \ge - \tfrac 12 \rho.
\end{split}
\label{regBO1}
\end{align}

\noi
Then, the modulated BO equation \eqref{BO}
on  $\T$
is 
locally well-posed in $H^s(\T)$
such that the solution $u$ belongs to the class
$ \cD_{w}^s([0, \tau] \times \T)$
defined in Remark \ref{REM:control}\,(i) below, 
where $\tau \in (0, T]$
denotes the local existence time.

\end{oldtheorem}

Theorem \ref{THM:1}
established regularization by noise
in several respects.
First of all, by taking $\rho$ sufficiently large, 
these  modulated equations
become locally well-posed
in $H^s(\T)$ for
any given $s \in \R$, 
even in the regime
where the unmodulated versions
are ill-posed.
See
\cite{KPV96, CKSTT03, 
KT1, M12,  KV19}
and 
\cite{Moli1, Moli2, 
Deng15, 
GKT23, KLV}
for the known well-\,/\,ill-posedness
results
for the (unmodulated) KdV equation:
\begin{equation}
\dt u+  \dx^3 u  =\dx u^2
\label{kdv2}
\end{equation}

\noi
and 
the (unmodulated) BO equation:
\begin{equation}
\label{BO2}
\dt u-    \H\dx^2 u  =\dx u^2,
\end{equation}

\noi
respectively.
The second point is the 
 semilinearization phenomenon.
The unmodulated 
KdV~\eqref{kdv2}
in $H^s(\T)$ for $s <  - \frac 12$
and
the unmodulated BO \eqref{BO2}
in $H^s(\T)$
 for  any $s \in \R$
are known to behave quasilinearly
such that 
the associated solution map 
is not smooth with respect to the $H^s$-topology;
see  \cite{BO97, MST, CCT,  GKT24}.
On the other hand, 
the construction of solutions in 
Theorem \ref{THM:1} 
is 
 based on a contraction argument
 such that the associated solution map
 is smooth (in fact, real analytic), 
thus  semilinearizing 
these equations
by adding  the modulation.

\begin{remark}\label{REM:control}\rm

(i)
Let $w$ be as in Theorem \ref{THM:1}.
Given $\tau > 0$, we
denote by 
\begin{align*}
 \cD_{w}^s([0, \tau] \times \T) \subset C([0, \tau];H^s(\T))
\end{align*}

\noi
the space of {\it paths controlled by $w$}, 
which are 
 paths $u \in C([0, \tau];H^s(\T))$ 
such that 
their modulated interaction representations 
$\uu(t)=\uw(t)^{-1}u(t)$
defined in \eqref{int1}
belong
to  $\CC^{\g}([0, \tau];H^s(\T))$; see~\eqref{Ho1}
for the definition of the $\CC^\g$-norm.
Then, 
uniqueness of solutions constructed in Theorem \ref{THM:1}
holds 
 in the class $\cD_{w}^s([0, \tau] \times \T)$, 
 where 
  $0 < \tau \le  T$ denotes the local existence time.
Namely, the uniqueness
does {\it not} hold
in the entire class 
$C([0, \tau];H^s(\T))$ (for $u$);
see \cite[Remark 1.4]{CGLLO1}.

\smallskip

\noi
(ii)
For $s < 0$, the quadratic nonlinearity $\dx u^2(t)$ does not make 
sense for fixed $t \ge0$.
As in the Fourier restriction norm method due to Bourgain \cite{BO93}
in the unmodulated case (namely, $w(t) = t$), 
the main point is to make sense of 
$\dx u^2$ as a space-time distribution
(and not for each fixed $t \ge 0$).
More precisely, in Theorem \ref{THM:1}\,(i), 
we give a meaning to  the nonlinearity (at the level 
of the modulated interaction representation $\uu$)
by making sense of 
the integral term
in 
\eqref{mild3x}
as a nonlinear Young integral
$\I^{X^\KDV}(\uu) $
as in \eqref{YZ1}
via the sewing lemma and \eqref{YQ6}, 
which requires temporal regularity $\g > \frac 12$
for $\uu$.
The main difficulty in proving unconditional uniqueness
comes from the
fact that we can {\it not} assume any positive temporal regularity
for~$\uu$.

\smallskip

\noi
(iii)
In proving Theorem \ref{THM:1}\,(i), 
we {\it a priori} assume that 
the modulated interaction representation $\uu$ of a solution 
belongs to $\CC^\g([0, \tau]; H^s(\T))$ for $\g > \frac 12$.
Noting that 
the modulated interaction representation  of 
any linear solution to the modulated KdV \eqref{kdv1}
is a constant-in-time function, belonging
to $\CC^\g([0, \tau]; H^s(\T))$ for any $\g > 0$, 
we may say that 
the a priori assumption
that 
$\uu \in \CC^\g([0, \tau]; H^s(\T))$ for $\g > \frac 12$
in proving Theorem \ref{THM:1}\,(i)
imposes that our solution be close to being a linear solution
just as in the Fourier restriction norm method due to Bourgain \cite{BO93}
in the unmodulated case (namely, $w(t) = t$);
see also \cite{Robert2, Robert3}
for the Fourier restriction norm method approach
to modulated dispersive PDEs.
Compare this with the situation in Theorem \ref{THM:UU1}\,(i)
for unconditional uniqueness, 
where we only assume that 
$\uu \in C([0, \tau]; H^s(\T))$
without further (temporal) regularity.

\end{remark}

\begin{remark}\rm 
In a recent work \cite{CGLLO2}, we studied
pathwise well-posedness of stochastic modulated dispersive equations
with multiplicative noises
and established a new 
regularization-by-noise phenomenon by exploiting the nonlinear interaction between the unknown and the noise.
Our strategy was 
to combine
the nonlinear Young integral approach in \cite{CGLLO1}
described above
with a suitable adaptation of 
the pathwise construction
of  stochastic convolutions
via 
the random tensor estimate (see \cite{DNY3, OWZ, OW3}) 
and the sewing lemma, 
which is a strategy
developed in \cite{CLO2, COZ}
to study
 stochastic (unmodulated) dispersive equations with multiplicative noises;
see also \cite {OS2}.
See \cite{CGLLO2} for details.

\end{remark}

\subsection{Unconditional uniqueness
of modulated dispersive PDEs}
\label{SUBSEC:UU}

In this subsection, 
we state our main results
on
unconditional well-posedness
(and also on nonlinear smoothing)
for modulated dispersive equations on the circle. 
Our first theorem is for the modulated KdV~\eqref{kdv1} on~$\T$.

We say that $u$
is a solution to the modulated KdV \eqref{kdv1} on $\T$, 
if it satisfies the following Duhamel formulation:
\begin{equation}
u(t) = U^w(t) u_0 +  U^w({t}) \int_0^t  U^w({t'})^{-1}\dx u^2(t')  dt'.
\label{mild1x}
\end{equation}

\noi
The same comment applies
to the other modulated dispersive equations considered in this paper.

\begin{theorem}\label{THM:UU1}
Given $\rho > \frac 54$,  $\frac12< \g < 1$, and $T> 0$, 
let  $w$ be $(\rho,\g)$-irregular on $[0, T]$ in the sense of Definition~\ref{DEF:ir}.

\smallskip

\noi
\textup{(i) (unconditional well-posedness).} 
Let   $s\ge 0$.
Then, the modulated KdV equation~\eqref{kdv1}
on  $\T$
is unconditionally locally well-posed in $H^s(\T)$.
In particular, 
if, in addition, $w$ is  $(\rho,\g)$-irregular  on $\R_+$, 
then the modulated KdV equation \eqref{kdv1}
on  $\T$
is unconditionally globally well-posed in $H^s(\T)$.

\smallskip

\noi
\textup{(ii) (nonlinear smoothing).} 
Suppose that 
$s, s_0 \in \R$ with 
$s_0 > s\ge 0$   satisfy
\begin{align*}
s_0 < s + 2\rho - \frac 52.
\end{align*}

\noi
Given $u_0 \in H^s(\T)$, 
let $u
\in C([0, \tau]; H^s(\T))$ be the  solution 
to the modulated KdV equation~\eqref{kdv1}
on $\T$
with $u|_{t = 0} = u_0$
constructed in Part (i), 
where 
$\tau \in (0,  T]$ denotes the local existence time.
Then, we have 
\begin{align*}
 u - e^{-w(t) \dx^3}u_0 \in C([0, \tau]; H^{s_0}(\T)).
\end{align*}

\end{theorem}

Theorem \ref{THM:UU1}\,(i) establishes the first unconditional uniqueness
result for modulated dispersive PDEs.
As a corollary to Theorem \ref{THM:UU1}\,(i)
with Theorem \ref{THM:A}, we have
the following unconditional well-posedness
result for the modulated KdV \eqref{kdv1} on $\T$, 
where a modulation is given by a fractional Brownian motion.

\begin{corollary}\label{COR:1}
Let $\{W_t\}_{t\in \R_+}$ be a fractional Brownian motion 
with Hurst index $0 < H < \frac 25$.
Then, 
with probability one, 
the modulated KdV equation \eqref{kdv1}
on  $\T$
with the modulation given by 
a sample path of the
 fractional Brownian motion 
 $\{W_t\}_{t\in \R_+}$
is unconditionally globally well-posed in $H^s(\T)$
for any $s \ge 0$.

\end{corollary}

\begin{remark}\rm
(i)
Similar results hold for each of the modulated dispersive equations considered in this paper, with the appropriate choice of exponents, when the modulation function is given by a fractional Brownian motion, or,  thanks to Theorem \ref{THM:B}, when the modulation function is given by a generic (in the sense of prevalence) H\"older function.

\smallskip

\noi
(ii)  Theorem \ref{THM:UU1}
holds only for $\rho > \frac 54$
(and thus for $H < \frac 25$) and hence is not applicable
to the case when a modulation $w$ is a Brownian motion
(corresponding to $H= \frac 12$)
which would require us to consider $\rho < 1$
in view of Theorem \ref{THM:A}.
When $\rho < 1$, 
by considering the contribution $|n| \sim |n_1| \gg |n_2|, |n_3|$
in \eqref{nox3}, 
it is easy to see that the crucial trilinear estimate
(Lemma  \ref{LEM:nonlin2} for any  $s = s_0 \in \R$)
fails
and thus our normal form approach for the modulated KdV \eqref{kdv1}
completely breaks down.

\end{remark}

\begin{remark}\label{REM:sharp1} \rm
(i)
In view of  the quadratic nonlinearity in \eqref{kdv1}, 
the regularity threshold $s \ge 0$ 
in Theorem~\ref{THM:UU1}
is {\it sharp} since the $L^2$-regularity is required to make
sense of $\dx u^2(t)$ as a spatial distribution for given $t \ge 0$
({\it without} intersecting with an auxiliary function space).
Namely, if we {\it only} assume that
\[ \uu \in C([0, \tau]; H^s(\T))\]

\noi
without further (temporal) regularity, 
the regularity $s = 0$ is the lowest possible value 
in making sense of the quadratic nonlinearity.
The same comment on sharpness of the $L^2$-regularity
also applies to the other modulated dispersive equations
considered in this paper.
We note that, in order to make sense of the nonlinearity $\dx u^2$
in $H^s(\T)$ for $s < 0$, we need to 
assume
some temporal regularity on $u$
(or rather its modulated interaction representation $\uu$)
and give a meaning to $\dx u^2$ as a space-time distribution
(in particular, $\dx u^2(t)$ does not make sense
for fixed $t \ge 0$)
just as in the Fourier restriction norm method due to Bourgain \cite{BO93}
in the unmodulated case (namely, $w(t) = t$)
or as in Theorem \ref{THM:1}, 
where we need to {\it a priori} assume that $\uu \in \CC^\g([0, \tau]; H^s(\T))$; 
see Remark \ref{REM:control}.

\smallskip

\noi
(ii)
We note that 
the condition $\rho > \frac 54$ in Theorem \ref{THM:UU1}\,(i)
is more restrictive than that in Theorem \ref{THM:1}\,(i).
This is due to the fact that, 
in Theorem \ref{THM:UU1}\,(i), 
we only assume  that
$\uu \in C([0, \tau]; H^s(\T))$, 
which is essential in proving unconditional uniqueness, 
while, in  Theorem~\ref{THM:1}\,(i), we assume that 
$\uu$ has additional temporal regularity
(namely, 
$\uu \in \CC^\g([0, \tau]; H^s(\T))$), 
which only yields conditional uniqueness.

Similarly, while the condition on $\rho$, $s$, and $s_0$
in Theorem \ref{THM:UU1}\,(ii) is more restrictive than
the corresponding result
on the nonlinear smoothing in \cite[Theorem 1.15]{CGLLO1}, 
our argument does not rely
on any auxiliary function space such as 
$\cC^\g([0, \tau]; H^s(\T))$
(even for the modulated interaction representation).
\end{remark}

We prove Theorem \ref{THM:UU1}
by implementing a normal form reduction
in the current modulated setting.
In a pioneering work \cite{BIT}, 
Babin, Ilyn, and Titi
implemented  a normal form argument
for the (unmodulated) KdV \eqref{kdv2}
and proved  its unconditional well-posedness 
in $L^2(\T)$.\footnote{In \cite{BIT}, 
Babin, Ilyn, and Titi did not make a connection
of their construction of solutions via ``differentiation by parts''
with 
normal form reductions.
Such a connection was made in an explicit manner
in a later work~\cite{GKO}.}
We point out that, while 
two normal form reductions
were needed in the unmodulated setting~\cite{BIT}, 
we  need to perform a normal form reduction
{\it only once} to prove Theorem \ref{THM:UU1}
for the modulated KdV
thanks to  the presence of the $(\rho, \g)$-irregular modulation function $w$,
which can be viewed as a manifestation
of regularization by noise.
In Subsection \ref{SUBSEC:NF}, we go over basic steps of
the normal form reduction.
We present a proof of Theorem \ref{THM:UU1}
in Section \ref{SEC:KDV}.

\medskip

Next, we state our main  result 
for the modulated BO \eqref{BO} on $\T$.

\begin{theorem}\label{THM:UU2}
Given $\rho > \frac 52$,  $\frac12< \g < 1$, and $T> 0$, 
let  $w$ be $(\rho,\g)$-irregular on $[0, T]$ in the sense of Definition~\ref{DEF:ir}.

\smallskip

\noi
\textup{(i) (unconditional well-posedness).} 
Let   $s\ge 0$.
Then, the modulated BO equation~\eqref{BO}
on  $\T$
is unconditionally locally well-posed in $H^s(\T)$.
In particular, 
if, in addition, $w$ is  $(\rho,\g)$-irregular  on $\R_+$, 
then the modulated BO equation \eqref{BO}
on  $\T$
is unconditionally globally well-posed in $H^s(\T)$.

\smallskip

\noi
\textup{(ii) (nonlinear smoothing).} 
Suppose that 
$s, s_0 \in \R$ with 
$s_0 > s\ge 0$   satisfy
\begin{align*}
s_0 \le  \rho - 1
\qquad \text{and}\qquad 
s_0 < s + \rho - \frac 52.
\end{align*}

\noi
Given $u_0 \in H^s(\T)$, 
let $u
\in C([0, \tau]; H^s(\T))$ be the  solution 
to the modulated BO equation~\eqref{BO}
on $\T$
with $u|_{t = 0} = u_0$
constructed in Part (i), 
where 
$\tau \in (0,  T]$ denotes the local existence time.
Then, we have 
\begin{align*}
 u - e^{w(t) \H\dx^2}u_0 \in C([0, \tau]; H^{s_0}(\T)).
\end{align*}

\end{theorem}

%

For the unmodulated BO \eqref{BO2}, 
Mo\c{s}incat and Pilod
proved
its unconditional well-posedness
in $H^s(\T)$ for $s >  3 - \sqrt{\frac {33}4} \approx 0.128$;
see also \cite{Ki3}.
Thus, Theorem \ref{THM:UU2}\,(i)
exhibits a strong regularizing effect in the modulated setting, 
regarding  unconditional well-posedness.
Furthermore, as in the modulated KdV case, 
we only need to carry out a normal form reduction
once for the modulated BO \eqref{BO}, whereas, for the unmodulated BO \eqref{BO2}, 
one first needs to apply a gauge transform 
and then carry out normal form reductions twice.
This  can be once again viewed as a manifestation
of regularization by noise.

We present a proof of Theorem \ref{THM:UU2}
in Section \ref{SEC:BO}.

\begin{remark}\rm
Given $0 < \dl < \infty$, 
 consider the 
modulated 
intermediate long wave equation (ILW) on the circle: 
\begin{equation}
\label{ILW1}
\dt u-    \Gdl\dx^2 u \cdot \dt w =\dx u^2, 
\end{equation}

\noi
where the dispersion operator~$\Gdl \dx^2$ is given 
by 
\begin{align*}
\ft{\Gdl \dx^2 f}(n) =
i p_\dl(n)  \ft f(n)
:=
 \ind_{n\ne 0}\cdot i \bigg(n^2 \coth(\dl n)  - \frac{n}{\dl }\bigg) \ft{f}(n), \quad n\in\Z.
\end{align*}

\noi
See
\cite{S19, Li24, LOZ, CLOP, CFLOP, FLZ, CLO}
for the known well-\,/\,ill-posedness
results
for the (unmodulated) ILW:
\begin{equation*}
\dt u-    \Gdl\dx^2 u  =\dx u^2.
\end{equation*}

Recall from  \cite[(5.22)]{CGLLO1}
that  the resonance function for ILW 
given by 
\begin{align*}
\Xi_\dl (\bar n)
& = \Xi_\dl(n,n_1,n_2) = - p_\dl(n)  +   p_\dl(n_1) +  p_\dl(n_2)
\end{align*}

\noi
satisfies 
\begin{align}
\Xi_\dl(\bar n) =  \Xi_\BO(\bar n) + O(\dl^{-2}), 
\label{IL8}
\end{align}

\noi
where
$\Xi_\BO(\bar n) $
is the resonance function for BO
defined in \eqref{BX4} below.
Then, together with the fact that 
$|\Xi_\BO(\bar n)| \ges \max(|n|, |n_1|, |n_2|)$
under $n = n_1 + n_2$ and $n n_1 n_2 \ne 0$, 
the relation~\eqref{IL8}
states
that high frequency (multilinear) behavior
of (modulated) ILW is essentially given by that of 
(modulated) 
BO.
In~\cite{CGLLO1}, 
we proved well-posedness
of the modulated ILW \eqref{ILW1}
under the same hypothesis as that for the modulated BO
by using this observation.
Similarly, 
 a slight modification of the proof
of Theorem \ref{THM:UU2}
 with \eqref{IL8}
yields the corresponding result
for the modulated ILW \eqref{ILW1}.
We omit details.

\end{remark}

Lastly, we consider the following modulated derivative 
nonlinear Schr\"odinger equation (dNLS) on the circle:
\begin{align}
i \dt u + \dx^2 u \cdot \dt w =  \dx u^2.
\label{dNLS1}
\end{align}

\begin{theorem}\label{THM:UU3}
Given $\rho > \frac 52$,  $\frac12< \g < 1$, and $T> 0$, 
let  $w$ be $(\rho,\g)$-irregular on $[0, T]$ in the sense of Definition~\ref{DEF:ir}.

\smallskip

\noi
\textup{(i) (unconditional well-posedness).} 
Let   $s\ge 0$.
Then, the modulated dNLS equation~\eqref{dNLS1}
on  $\T$
is unconditionally locally well-posed in $H^s(\T)$.

\smallskip

\noi
\textup{(ii) (nonlinear smoothing).} 
Suppose that 
$s, s_0 \in \R$ with 
$s_0 > s\ge 0$   satisfy
\begin{align*}
s_0 \le  \rho - 1
\qquad \text{and}\qquad 
s_0 < s + \rho - \frac 52.
\end{align*}

\noi
Given $u_0 \in H^s(\T)$, 
let $u
\in C([0, \tau]; H^s(\T))$ be the  solution 
to the modulated dNLS equation~\eqref{dNLS1}
on $\T$
with $u|_{t = 0} = u_0$
constructed in Part (i), 
where 
$\tau \in (0,  T]$ denotes the local existence time.
Then, we have 
\begin{align*}
 u - e^{i w(t) \dx^2}u_0 \in C([0, \tau]; H^{s_0}(\T)).
\end{align*}

\end{theorem}

In \cite{CGKO}, Chung, Guo, Kwon, and the fourth author
studied the unmodulated version of~\eqref{dNLS1} on the circle:
\begin{align}
i \dt u + \dx^2 u =  \dx u^2.
\label{dNLS0}
\end{align}

\noi
By implementing an infinite iteration
of normal form reductions introduced in \cite{GKO}
and combining it with a gauge transform
(the Hopf-Cole transformation), 
they proved
small data unconditional global well-posedness in $L^2(\T)$.
We point out that, in the modulated setting, 
 Theorem~\ref{THM:UU3}
follows from one normal form reduction.
Note also that 
 Theorem \ref{THM:UU3}\,(i)
claims only local well-posedness;
see \cite[Appendix A]{CGKO}
for an explicit construction of
a finite-time blowup solution
to the (unmodulated) dNLS \eqref{dNLS0}.
We present 
 a proof of Theorem \ref{THM:UU3}
in  Section~\ref{SEC:NLS}.

\subsection{Normal form approach
for the modulated KdV}
\label{SUBSEC:NF}

In this subsection, we sketch the main idea
of the normal form approach to construct solutions
to modulated dispersive equations
by taking the modulated KdV \eqref{kdv1} as an example.
In the unmodulated setting, 
this normal form approach was introduced  by Babin, Ilyn, and Titi \cite{BIT}
and was further developed in \cite{KO, GKO, OW, Ki2}.
See also Remark \ref{REM:NF1}.

In \cite{CGLLO1}, 
the key step  in proving Theorem \ref{THM:1}\,(i)
is to show the boundedness
of the bilinear driver 
$X^\KDV_{t,r}$ in \eqref{K2}
from 
 $\big(H^s(\T)\big)^{\otimes 2}$ into $H^s(\T)$
with the norm $\les |t-r|^\g$, 
where the following bound:
\begin{align}
| \Phi^w_{t,r}(\Xi_\KDV (\bar n))|
\le
\|\Phi^w\|_{  \W^{\rho,\g}_T} |t - r|^\g \jb{\Xi_\KDV (\bar n)}^{-\rho}, 
\label{decay1}
\end{align}

\noi
coming from 
\eqref{rho1},  
plays a crucial role.  See \cite[Proposition 4.1]{CGLLO1}.
Since $\g > \frac 12$, 
this then allows us to apply the sewing lemma
and make sense of the integral term in \eqref{mild3x}
as a nonlinear Young integral.
In proving Theorem \ref{THM:UU1}\,(i)
where we do not impose any positive temporal regularity,
such an approach immediately breaks down. 
In order to overcome
this difficulty, 
we proceed with a normal form reduction (namely, 
integration by parts in time) 
together with the identity \eqref{NS1} below,
which transforms the original equation
\eqref{mild3x} (at the level of the modulated interaction representation)
into the equation
\eqref{NF4x} (see also \eqref{K2} and \eqref{nox1})
with the quadratic and {\it cubic} nonlinearities.
The main point is that 
these quadratic and cubic nonlinearities
 come with 
the ``good'' factor $\Phi^w_{t,r}(\Xi_\KDV (\bar n))$
for which \eqref{decay1} applies, 
allowing us to establish nonlinear estimates
{\it without} imposing any positive temporal regularity.
Heuristically speaking, 
the normal form reduction 
converts  the issue of having no  positive temporal regularity
into a higher order (namely, cubic in this case)  nonlinearity
(which is still manageable under a suitable assumption).

By taking the Fourier transform of \eqref{kdv1a}, we have
\begin{align}
\dt \ft \uu(t, n) 
= in \sum_{ \substack{n_1, n_2 \in \Z_*\\n = n_1+n_2}}
  e^{i \Xi_\KDV (\bar n)w(t)}
\ft \uu(t, n_1)  \ft \uu(t, n_2), \quad n \in \Z_*, 
\label{kdv1b}
\end{align}

\noi
where $\Xi_\KDV (\bar n) $ is 
the resonance function  in \eqref{K3a}.
In the unmodulated setting (i.e.~$w(t) = t$), 
the first step of normal form reductions
in \cite{BIT} was based
on the identity:
\begin{align}
e^{i \Xi_\KDV (\bar n) t} = \frac{\dt e^{i \Xi_\KDV (\bar n) t}}{i \Xi_\KDV (\bar n)}.
\label{NS0}
\end{align}

\noi
In the current modulated setting, 
we replace this identity by the following
simple, but crucial observation:
\begin{align}
\dd_{t_2}  \Phi^w_{t_1, t_2}(\Xi_\KDV (\bar n)) = - e^{i \Xi_\KDV (\bar n)w(t_2)}, 
\label{NS1}
\end{align}

\noi
which follows from  \eqref{rho2} and the fundamental theorem of calculus.
Then, proceeding with a formal computation,\footnote{Namely, 
we do not worry about justifying computations
such as switching the sum
and the integration
in the formal computation presented here.
We justify all the formal steps in Section \ref{SEC:KDV}.}
a normal form reduction (namely, 
integration by parts in time, using
\eqref{NS1})
yields
\begin{align}
\ft \uu(t, n) - \ft \uu(0, n) 
& = i n  \int_0^t \sum_{\substack{n_1, n_2 \in \Z_*\\n = n_1 + n_2}}
e^{i  \Xi_\KDV (\bar n)w(t')} \ft \uu(t', n_1)\ft \uu(t', n_2) dt'
\notag\\\
& =-  i n  \sum_{\substack{n_1, n_2 \in \Z_*\\n = n_1 + n_2}}
 \int_0^t
 \partial_{t'} \Phi^w_{t, t'}(\Xi_\KDV (\bar n))\ft \uu(t', n_1)\ft \uu(t', n_2) dt'
 \notag\\
 & =
  i n  \sum_{\substack{n_1, n_2 \in \Z_*\\n = n_1 + n_2}}
\Phi^w_{t, 0}(\Xi_\KDV (\bar n))\ft \uu(0, n_1)\ft \uu(0, n_2) 
\label{NF2}
\\
&  \quad + 2 i n  \sum_{\substack{n_1, n_2 \in \Z_*\\n = n_1 + n_2}}
 \int_0^t
 \Phi^w_{t, t'}(\Xi_\KDV (\bar n))
 \dt \ft \uu(t', n_1) \ft \uu(t', n_2) dt' \notag\\
& =: \F_x\big(X^\KDV_{t, 0}(\uu(0), \uu(0))\big)(n)   + A(t, n), 
\notag
\end{align}

\noi
where $X^\KDV_{t, 0}$ is the bilinear driver defined in 
\eqref{K2}.
At the third step, we used 
the fact that $\Phi^w_{t, t} = 0$, 
the product rule for differentiation, and 
symmetry in $n_1$ and $n_2$. 
See Lemma \ref{LEM:nonlin1} below
for the boundedness and nonlinear smoothing property
of the bilinear driver $X^{\KDV}_{t, 0}$.

\begin{remark}\label{REM:NFx}
\rm
We point out 
that, after the integration by parts in \eqref{NF2},  there is 
no boundary term from the right endpoint  since $\Phi^w_{t, t} = 0$.
This  is different from a usual 
application of a
normal form reduction in the unmodulated setting, 
where we also have a non-trivial contribution from 
the right endpoint.
Furthermore, 
in the usual construction of solutions via the normal form approach, 
the boundary terms do not become small even if we take
a short time interval, 
and thus we need to introduce a cutoff parameter
(based on a frequency or modulation size)
to create smallness to carry out a contraction;
see \cite{BIT, KO, GKO}.
In the current modulated setting, however, 
the estimate \eqref{non1c} for the boundary term comes with the factor~$t^\g$
and thus there is no need to introduce such a parameter.

Our version of the normal form reduction, 
using the identity \eqref{NS1}, 
also applies to the unmodulated setting, 
giving a simplification of the argument
in the sense described above
(namely, no need to introduce a parameter
to create a small factor on the boundary terms).

\end{remark}

Given $0 \le r <  t \le T$, 
define the trilinear operator $\NN^\KDV_{t, r}(f_1, f_2, f_3)$,  acting on functions on~$\T$, 
 by 
\begin{align}
\begin{split}
& \F_x\big(\NN^\KDV_{t, r}(f_1, f_2, f_3)\big)(n)\\
& \quad =   
- 2 n  \sum_{\substack{n_1, n_2, n_3 \in \Z_*\\n = n_{123}}}
 \Phi^w_{t, r}(\Xi_\KDV (n, n_{12}, n_3))
e^{i  \Xi_\KDV (n_{12}, n_1, n_2)w(r)}n_{12} 
\prod_{j = 1}^3\ft f_j(n_j), 
\end{split} 
\label{nox1}
\end{align}

\noi
where   we used the following short-hand notation:
\begin{align}
n_{1\cdots k} = n_1 + \cdots + n_k.
\label{no1}
\end{align}

\noi
See Lemma \ref{LEM:nonlin2} below
for the boundedness and nonlinear smoothing property
of the trilinear operator $\NN_{t, r}^\KDV$.

By substituting~\eqref{kdv1b}
into $A(t, n)$
on the right-hand side of \eqref{NF2},  we have
\begin{align}
A(t, n) = \int_0^t 
\F_x\big(\NN_{t, t'}^\KDV(\uu(t'), \uu(t'), \uu(t'))\big)(n)dt'.
\label{NF3}
\end{align}

\noi
Then, putting 
\eqref{NF2} and \eqref{NF3} together, 
we have
\begin{align*}
\ft \uu(t, n) - \ft \uu(0, n) 
 = \F_x\big(X^\KDV_{t, 0}(\uu(0)  )\big)(n)
 +  \int_0^t 
\F_x\big(\NN^\KDV_{t, t'}(\uu(t'))\big)(n)dt', 
\end{align*}

\noi
where we used the  short-hand notation \eqref{short1}.
Namely, we have
\begin{align}
\begin{split}
\uu(t) 
 =  u_0 + 
 X^\KDV_{t, 0}(u_0 )
+ 
\int_0^t 
\NN^\KDV_{t, t'}(\uu(t'))dt'.
\end{split}
\label{NF4x}
\end{align}

\noi
Hence, 
we arrive at the following {\it normal form equation}
for the 
modulated KdV \eqref{kdv1}:
\begin{align}
\begin{split}
u(t) 
& = \uw(t) u_0 + 
\uw(t) X^\KDV_{t, 0}(u_0 )
+ \uw(t)
\int_0^t 
\NN^\KDV_{t, t'}(\uw(t')^{-1}u(t'))dt'.
\end{split}
\label{NF4a}
\end{align}

\noi
The main point of the normal form reduction
is to transform the original
equation~\eqref{kdv1a}
to the normal form equation \eqref{NF4x}
which 
encodes multilinear dispersive smoothing
(coming from the irregularity of $w$)
in an explicit manner
thanks to the presence of 
$\Phi^w_{t, r}(\Xi_\KDV)$
in $X^\KDV_{t, 0}$
and $\NN^\KDV_{t, t'}$
(see 
\eqref{K2}
and
\eqref{nox1}).
Using the boundedness properties
of $X^\KDV_{t, 0}$
and $\NN^\KDV_{t, t'}$
(Lemmas \ref{LEM:nonlin1}
and \ref{LEM:nonlin2}), 
 a standard contraction argument
yields
unconditional local
well-posedness
of \eqref{NF4x} (and thus of \eqref{NF4a})
in $H^s(\T)$, $s \ge 0$, 
under the hypothesis of 
Theorem~\ref{THM:UU1}.
In the proof of 
Theorem \ref{THM:UU1} presented
in Section \ref{SEC:KDV}, 
we justify the formal computation 
in \eqref{NF2}
and show that 
any solution to the modulated KdV \eqref{kdv1}
(namely, satisfying the Duhamel formulation \eqref{mild1x})
in the class $C([0, \tau]; L^2(\T))$
is also a solution to the normal form equation~\eqref{NF4a}, 
which 
in turn allows us
to show that 
a unique solution
to the normal form equation \eqref{NF4a}
is also an unconditionally unique solution to  the original modulated KdV~\eqref{kdv1}.

\begin{remark}\label{REM:NF9}
\rm

(i) 
Consider 
a modulated dispersive PDE \eqref{ME1}, 
where the nonlinearity $\NN(u)$
is $k$-linear.
Let $n_j$ denote the (spatial) frequency 
of the $j$th factor in $\NN(u)$
and set 
\begin{align}
n_{\max} = \max_{j = 1, \dots, k} |n_j|.
\label{max1}
\end{align}

\noi
If the equation is {\it non-resonant}
with a good lower bound on the resonance function
such as  
$|\Xi(\bar n)|\ges n_{\max}^\ta$
for some (small) $\ta > 0$, 
then
a slight modification of the proof of Theorem \ref{THM:UU1}
yields 
unconditional local well-posedness
by one normal form reduction, 
at least when the modulation is sufficiently irregular.
See Remark \ref{REM:NLS}
and Appendix \ref{SEC:A}
for a discussion on the modulated cubic NLS \eqref{NLS2}
 which is {\it not} non-resonant.

\smallskip

\noi
(ii) 
By adapting the normal form approach 
developed in 
\cite{KOY, Ki2, FO, CLOZ}
for (unmodulated) dispersive equations on $\R^d$, 
one can study unconditional well-posedness of modulated 
dispersive equations on $\R^d$.
We, however, do not pursue this issue here.

\end{remark}

\begin{remark}\label{REM:NF1} \rm
The normal form reduction
presented in this subsection
corresponds to the so-called Poincar\'e normal form reduction, 
where there is no resonance.
When there are resonances, it is called
the Poincar\'e-Dulac normal form reduction;
see \cite[Section 1]{GKO}.
See also Appendix 
 \ref{SEC:A}, 
where 
we implement a Poincar\'e-Dulac normal form reduction for the modulated cubic NLS.
In the unmodulated setting, 
the Poincar\'e-Dulac normal form reductions have been 
used in various contexts:
unconditional uniqueness \cite{BIT, KO, GKO, OW, Ki2, Ki3, MP2, OSW}, 
nonlinear smoothing and the nonlinear Talbot effect \cite{ETz, ETz3, ETz2}, 
existence of global attractors~\cite{ETz3}, 
reducibility~\cite{CGKO}, 
improved energy estimates 
in both the deterministic and probabilistic contexts
\cite{TT, GO, OW2, OST, OS, STz}.
We also mention a very recent work 
\cite{CLOZ}, 
where an infinite iteration of normal form reductions
is combined with a complete integrability approach
to establish shallow-water convergence of the
(scaled) ILW to KdV in $L^2(\M)$, $\M = \T$ or $\R$.
We point out that 
the
higher order $I$-method (with correction terms) \cite{CKSTT03, CKSTT4}
also corresponds to the Poincar\'e-Dulac normal form reduction, 
but applied to the equation satisfied by a modified energy.
It may be of interest to investigate 
analogues of these results
in the modulated setting.

\end{remark}

\begin{remark}\rm

Instead of \eqref{NS1}, we may use the following identity:
\begin{align*}
\dd_{t_1} \Phi^w_{t_1, t_2}(\Xi_\KDV (\bar n)) = e^{i \Xi_\KDV (\bar n)w(t_1)}
\end{align*}

\noi
and proceed with a normal form reduction.
By repeating the computation presented in this subsection, we
then obtain
a slightly different normal form equation:
\begin{align}
\begin{split}
\uu(t) 
 =  u_0 + 
 X^\KDV_{t, 0}(\uu (t) )
+ 
\int_0^t 
\NN^\KDV_{0, t'}(\uu(t'))dt'.
\end{split}
\label{NF4y}
\end{align}

\noi
Note that the argument of $\uu (t)$
of $X^\KDV_{t, 0}$ is now evaluated at the right endpoint of the interval
$[0, t]$.
Compare this with \eqref{NF4x}.
While we prove Theorem \ref{THM:UU1} 
using \eqref{NF4a}, 
one may equally use the normal form equation \eqref{NF4y}
to prove Theorem \ref{THM:UU1} .
We, however, point out that
the normal form reduction with \eqref{NS1}
allows us to directly compare our normal form approach 
with the sewing lemma approach to nonlinear Young integrals
developed in \cite{CGLLO1}.
See the next subsection.

\end{remark}

\begin{remark}\label{REM:NLS}\rm
Consider the following modulated cubic NLS on $\T$:
\begin{align}
i \dt u + \dx^2 u \cdot \dt w= |u|^2u.
\label{NLS2}
\end{align}

\noi
The associated resonance function 
$\Xi_\NLS (\bar n)$  is given by 
\begin{align}
\begin{split}
\Xi_\NLS (\bar n)
& = \Xi_\NLS(n,n_1,n_2, n_3) = 
n^2 - n_1^2 + n_2^2 - n_3^2\\
& = 2(n- n_1) (n - n_3), 
\end{split}
\label{Xi9}
\end{align}

\noi
where the last step holds under $n = n_1 - n_2 + n_3$, 
and thus we see that 
 the resonance function 
$\Xi_\NLS (\bar n)$
 vanishes when $n = n_1$ or $n_3$.
Namely, the modulated cubic NLS \eqref{NLS2} is 
{\it not} non-resonant and thus the observation made in 
Remark \ref{REM:NF9}\,(i) 
does not apply.

In \cite{GKO, OW}, 
the fourth author with his co-authors 
studied a normal form approach to 
 the unmodulated cubic NLS  on $\T$:
\begin{align}
i \dt u + \dx^2 u = |u|^2u.
\label{NLS3}
\end{align}

\noi
In order to overcome
the difficulty coming from  the nearly resonant contributions: $|\Xi_\NLS(\bar n)| \ll n_{\max}$, 
where $n_{\max}$ is as in \eqref{max1},
the authors introduced 
 an 
 infinite iteration of normal form reductions, 
which  transformed \eqref{NLS3}
 into the normal form equation, 
involving infinite series of nonlinearities of arbitrarily high degrees.
This allowed them to prove 
sharp unconditional uniqueness 
of  \eqref{NLS3} in $H^s(\T)$, $s \ge \frac 16$, 
and also unconditional uniqueness
of the normal form equation in $L^2(\T)$
and 
in almost critical Fourier-Lebesgue spaces;
see~\cite{OW}
for further details.

In Appendix \ref{SEC:A}, 
we implement a (Poincar\'e-Dulac) normal form reduction for the modulated cubic NLS \eqref{NLS2}, 
using the identity 
\eqref{NS1},
and prove its unconditional uniqueness;
 see Theorem \ref{THM:UU4}. 
See also Remark \ref{REM:NLS1}
for a discussion on 
a general model which includes (nearly) resonant interactions.

\end{remark}

\subsection{Comparison with 
the sewing lemma approach}
\label{SUBSEC:comp}

In this subsection, we
compare  the normal form approach 
presented in the previous subsection  with the sewing lemma approach 
to nonlinear Young integrals in \cite{CGLLO1}.
For this purpose, we first need to introduce some preliminary notations.

Let $V$ be a Banach space and $T>0$.
For $n\in\N$, we denote 
\begin{align*}
\Delta_{n, T} = 
\big\{ (t_1, \ldots, t_n) \in [0,T]^n: \ t_i > t_j 
\text{ for } i < j\big\}.
\end{align*}
We denote by $C_{n,T}V$ 
the space of continuous functions 
from $\Delta_{n,T}$ to $V$. 
We define the coboundary operator 
$\updl: C_{n,T} V  \to C_{n+1,T} V$ 
as follows; 
given  $f\in C_{n,T} V$ 
and $(t_1,\ldots, t_{n+1}) \in \Dl_{{n+1},T}$, 
we set
\begin{align*}
(\updl f)_{t_1,\ldots , t_{n+1}} = 
\sum_{k=1}^{n+1} (-1)^{k} f_{t_1, \ldots, t_{k-1}, t_{k+1}, \ldots, t_{n+1}}.
\end{align*}

\noi
For example, for $f\in C_T V$
and 
$g\in C_{2,T} V$, 
we have 
\begin{align}
\begin{split}
(\updl f )_{t,r} &= f_t - f_r, \\
(\updl g)_{t_1,t_2,t_3} &= 
g_{t_1,t_3} - g_{t_1,t_2} - g_{t_2,t_3}
\end{split}
\label{dl1}
\end{align}

\noi
for $(t,r)\in\Dl_{2,T}$
and $(t_1,t_2,t_3) \in \Dl_{3,T}$.
As noted in \cite{GT10}, 
the sequence 
\begin{align*}
0 \too \R \too C_{1, T}V \stackrel{\updl}{\too} C_{2, T}V
\stackrel{\updl}{\too} C_{3, T}V
\stackrel{\updl}{\too} \cdots
\end{align*}

\noi
is exact. 
In particular, we have 
 $\updl\circ\updl =0$ and 
if $f \in C_{n,T} V$ with $\updl f =0$, 
then there exists a $g\in C_{n-1,T}V$ 
such that $f= \updl g$; 
see, for example,  \cite[Lemma 2.1]{GT10}. 

Given $0 < \g \le 1$, we denote by $C^\g_T V = C^\g([0, T]; V)$ the space of $\g$-H\"older continuous functions taking values in $V$, endowed  with the seminorm:
\begin{align*}
\|f\|_{C^\g_T V} = \sup_{(t,r)\in \Dl_{2,T}} 
\frac{\|(\updl f)_{t,r}\|_V}{|t-r|^\g}.
\end{align*}

\noi
We also define 
 $\CC^\g_T V = \CC^\g([0, T]; V)$ via the norm:
\begin{align}
\| f \|_{\CC^\g_TV} = \| f \|_{L^\infty_TV} + \|f\|_{C^\g_T V}.
\label{Ho1}
\end{align}

\noi 
We also introduce the spaces 
$C^\g_{n,T}V$, $n=2,3$, 
equipped with the following H\"older-type norms;
 for $g\in C_{2,T}V$ and $h\in C_{3,T}V$, we set
\begin{align}
\begin{split}
\| g\|_{C^\g_{2,T} V} & 
= \sup_{(t,r) \in \Dl_{2,T}} \frac{\|g_{t,r} \|_{V} }{|t-r|^{\g}}, \\
\| h\|_{C^\g_{3,T} V} & 
= \inf_{0<\al<\g} 
\sup_{(t_1,t_2,t_3) \in \Dl_{3,T}} 
\frac{ \|h_{t_1,t_2,t_3}\|_V}
{|t_1-t_2|^{\al} |t_2-t_3|^{\g-\al}}.
\end{split}
\label{Ho2}
\end{align}

\noi
Given a Banach space $V$ and $k \in \N$, 
we  use $\Lip_k(V)$ 
to denote the Banach space of 
locally Lipschitz maps $f:V\to V$
with polynomial growth of order $k$
such that 
\begin{align}
\|f\|_{\Lip_k(V)} = 
\sup_{x,y\in V} 
\frac{\|f(x)-f(y)\|_V}{\|x-y\|_V \big(1+\|x\|_V+\|y\|_V\big)^{k-1}}
<\infty .
\label{Lip1}
\end{align}

\noi
Then, 
from \eqref{Ho2} and \eqref{Lip1}, we have 
\begin{align}
\|f \|_{C^\g_{2, T}\Lip_k(V)}
= \sup_{(t,r) \in \Dl_{2,T}}
\frac 1{|t-r|^{\g}}
\sup_{x,y\in V} 
\frac{\| f_{t, r}(x)- f_{t, r}(y)\|_V}{\|x-y\|_V \big(1+\|x\|_V+\|y\|_V\big)^{k-1}}.
\label{Ho3}
\end{align}

\medskip
 
 Next, we 
review the essential idea
 on the construction of a nonlinear Young integral 
 via the sewing lemma (Lemma \ref{LEM:sew}) presented in \cite{CGLLO1}.
Fix a driver  $X\in C^\g_{2,T}\Lip_k(V)$, where $C^\g_{2,T}\Lip_k(V)$
is as in~\eqref{Ho3},  such that 
\begin{align*}
X_{t,r}(0)=0
\end{align*} 

\noi
for any $(t,r) \in  \Dl_{2,T}$, 
and 
\begin{align}
\label{Ja1}
(\updl X)_{t_1,t_2,t_3} = 0
\end{align}

\noi
for any $(t_1,t_2,t_3)\in \Dl_{3,T}$.
Given 
$\uu\in \CC^{\al}([0,T];V)$ for some $0 < \al < 1$, 
the nonlinear Young integral $\I^X(\uu)$ of $\uu$
is a unique function 
with $\I^X(\uu)(0) = 0$
whose increment is given by 
\begin{align}
\updl 
\big(\I^X(\uu)\big)_{t, r} 
= X_{t, r}(\uu_\bul).
\label{YQ1}
\end{align}

\noi
Here, $\uu_\bul$ denotes 
the function $\uu$ evaluated at the variable of integration
$\bul$, satisfying $r \le \bul \le  t$.
As it is, the expression on the right-hand side of \eqref{YQ1}
is not well defined.
By replacing $\bul$ by the left endpoint $r$, 
we have 
\begin{align}
X_{t, r}(\uu_\bul)
= X_{t, r}(\uu_r) + R_{t, r}
\label{YQ2}
\end{align}

\noi
for some two-parameter process $R = R^{X, \uu}$.
%
Define  $\Theta = \Ta^{X, \uu}$ on $\Dl_{2,T}$
by setting
\begin{align}
\Theta_{t,r} = X_{t,r}(\uu_r), \quad (t,r)\in \Dl_{2,T}.
\label{Ja1a}
\end{align}

\noi
Then, 
by applying the coboundary operator $\updl$
defined in \eqref{dl1}
to \eqref{YQ2}
and noting that the left-hand side of~\eqref{YQ2} vanishes
under $\updl$, 
any error term $R$ in~\eqref{YQ2} (if it exists) satisfies
\begin{align}
(\updl R)_{t_1, t_2, t_3}
= - (\updl  \Theta)_{t_1, t_2, t_3}
=  X_{t_1, t_2}( \uu_{t_2}) - 
  X_{t_1, t_2}( \uu_{t_3})
\label{YQ3}
\end{align}

\noi
for any $(t_1, t_2, t_3) \in\Dl_{3, T}$, 
where we used 
 \eqref{Ja1} at the second equality.

We now recall 
the  sewing lemma due to the first author \cite{Gub04};
see also
\cite[Proposition~2.3, Corollaries  2.4 and  2.5]{GT10}
and  \cite[Lemma~4.2]{FH20}.
We set 
$C^{1+}_{n,T}V = \bigcup_{\g>1}C^\g_{n,T}V$.

\begin{lemma}[sewing lemma]
\label{LEM:sew}

Let $V$ be a Banach space and 
fix $T>0$. 
Then,  there exists a unique linear map \textup{(}called the sewing map\textup{)}
$\Lambda:C^{1+}_{3,T} V \cap 
\Ker \updl|_{C_{3,T}V}
\to C^{1+}_{2,T}V$ such that 

\smallskip
\begin{enumerate}
\item[(i)] 
We have 
$\updl \Lambda h = h$
for each  $h\in C_{3,T}V\cap \Ker  \updl|_{C_{3,T} V}$.

\smallskip

\item[(ii)]
 For each $\z >1$,
the sewing map $\Lambda$ is continuous
from $C^\z _{3,T}V\cap 
\Ker  \updl|_{C_{3,T} V}$ to 
$C^\z _{2,T}V$ such that 
\begin{align*}
\|\Lambda h \|_{C^\z _{2,T}V} 
\le \frac{1}{2^\z - 2}  \| h \|_{C^\z _{3,T}V} 
\end{align*}

\noi
for any $h\in C^\z _{3,T}V$.

\smallskip
\item[(iii)] 
Given any  $g\in C_{2,T}V$ 
with 
$\updl g\in C^\z _{3,T}V$, 
 there exists  unique
$f\in C([0, T];V)$  \textup{(}modulo an additive  constant\textup{)} 
such that 
$\updl f = (\Id - \Lambda \updl)g$. 
In addition, 
we have 
\begin{align*}
(\updl f)_{t,r} = \lim_{|\Pi([r,t])|\to 0} 
\sum_{j=0}^n g_{t_j,t_{j+1}}
\end{align*}

\noi
 for any $(t,r)\in \Dl_{2,T}$,
 where 
 the limit is over any partition
 $\Pi([r,t])$  
 of  $[r,t]$\textup{:} 
\[\Pi ([r,t]) = \{r = t_n < \dots < t_1 <  t_0 = t\}\]
whose mesh size 
$|\Pi([r,t])| = \sup_{j} |t_j-t_{j+1}|$ 
tends to $0$.

\end{enumerate}

\end{lemma}

Since
$X\in C^\g_{2,T}\Lip_k(V)$
and 
$\uu\in \CC^{\al}([0,T];V)$, 
it follows from \eqref{YQ3} with \eqref{Ho3} and \eqref{Ho1} that 
$\updl R 
= -\updl \Ta 
 \in C^{\g + \al}_{3,T}V$, 
 where $\Ta$ is as in \eqref{Ja1a}; see \cite[(3.27)]{CGLLO1}.
Hence, if $\g + \al > 1$, 
then we can apply the sewing lemma (Lemma~\ref{LEM:sew})
to define an error term $R$ by setting
\begin{align}
R = - \Ld \updl  \Theta
\in C^{\g + \al}_{2,T}V.
\label{YQ4} 
\end{align}

\noi
Then, 
putting
 \eqref{YQ1},  
 \eqref{YQ2},  
and \eqref{YQ4} together,   
we define the nonlinear Young integral $\I^X(\uu)$ of $\uu$
(with respect to the nonlinear Young driver $X$)
to be  a unique function 
in $ \CC^\g([0,T]; V)$
with $\I^X(\uu)(0) = 0$
whose increment is given by 
\begin{align}
\updl (\I^X(\uu))
= (\Id - \Ld \updl )\Ta.
\label{YQ5}
\end{align}

\noi
Once \eqref{YQ5} holds, 
the nonlinear Young integral $\I^X(\uu)$ 
is given by 
the unique limit of 
Riemann-Stieltjes type sums:
\begin{align}
\I^X(\uu)(t) 
= \lim_{|\Pi([0,t])|\to 0} 
\sum^n_{j=0} \Theta_{t_j,t_{j+1}}
= \lim_{|\Pi([0,t])|\to 0} 
\sum^n_{j=0} X_{t_j,t_{j+1}}(\uu_{t_{j+1}}), 
\label{YQ6}
\end{align}

\noi
 where 
the limit is in the sense of Lemma \ref{LEM:sew}\,(iii).

\medskip

We are now ready to compare the normal form approach
with the sewing lemma approach.
Given  $\rho > \frac 54 $ and  $\frac 12  < \g < 1$, 
fix a 
 $(\rho,\g)$-irregular function $w$ on an interval $[0, T]$
 and $s\ge 0$ as in Theorem \ref{THM:UU1}.
Consider the modulated KdV \eqref{kdv1} on $\T$
and write 
\eqref{mild3x}
as
\begin{equation}
\uu = u_0 + \I^{X^\KDV}(\uu), 
\label{YZ1}
\end{equation}

\noi
Here, 
$X^\KDV$
is the bilinear driver defined in \eqref{K1}.
The main goal is to give a precise meaning
to 
the integral term $\I^{X^\KDV}(\uu)$.

\medskip

\noi
$\bul$ {\bf Sewing lemma approach:}
From 
Lemma \ref{LEM:nonlin1}, 
we have
$X\in C^\g_{2,T}\Lip_k(V)$.
Thus, 
if $\uu\in \CC^{\al}([0,T];H^s(\T))$
for some $0 < \al < 1$ such that $\g + \al > 1$, 
then
we can apply the sewing lemma 
(Lemma \ref{LEM:sew})
as described above
to construct
the nonlinear Young integral $\I^{X^\KDV}(\uu)$. 
From \eqref{YQ5} with \eqref{Ja1a}, we have 
\begin{align}
\I^{X^\KDV}(\uu)(t) - \I^{X^\KDV}(\uu)(r)
= X^\KDV_{t, r} (\uu(r)) - (\Ld \updl \Ta)_{t, r}
\label{YZ2}
\end{align}

\noi
for any $0 \le r < t \le T$, 
where $\Ld$ is the sewing map.

\medskip

\noi
$\bul$ {\bf Normal form approach:}
Let $\uu \in C([0, \tau]; H^s(\T))$
be the unique solution to \eqref{YZ1}
constructed in Theorem \ref{THM:UU1}.
Then, by repeating the computations in \eqref{NF2}
over an time interval $[r, t]$, 
we have
\begin{align}
\begin{split}
\I^{X^\KDV}(\uu)(t) - \I^{X^\KDV}(\uu)(r)
& = \uu(t) - \uu(r) \\
&  =  
 X^\KDV_{t, r}(\uu (r) )
+ 
\int_r^t 
\NN^\KDV_{t, t'}(\uu(t'))dt'
\end{split}
\label{YZ3}
\end{align}

\noi
for any $0 \le r < t \le \tau$, 
where $\NN^\KDV_{t, t'}$ is as in \eqref{nox1}.
We point out that in obtaining \eqref{YZ3}, 
we used the fact that $\uu$ is a solution to \eqref{YZ1}.

\medskip

By comparing \eqref{YZ2} and \eqref{YZ3}, 
we obtain 
\begin{align}
  (\Ld \updl \Ta)_{t, r}
 = 
- \int_r^t 
\NN^\KDV_{t, t'}(\uu(t'))dt'
\label{YZ4}
\end{align}

\noi
for any $0 \le r < t \le \tau$, 
where
the left-hand side makes sense
only for  $\uu\in \CC^{\al}([0,\tau];H^s(\T))$
with $\g + \al > 1$, 
whereas 
the right-hand side makes sense
for  $\uu\in C([0,\tau];H^s(\T))$.
Hence, we can view 
the right-hand side
of \eqref{YZ4}
as an extension of 
the definition of 
$\Ld \updl \Ta$ on the left-hand side
to functions of lower temporal regularity $\al \ge 0$, 
provided that $\uu$ is a solution to \eqref{YZ1}.
Namely, 
in the current modulated setting, 
{\it the normal form reduction
with the (controlled) structure~\eqref{YZ1}
extends 
the construction of the nonlinear Young integral
$\I^{X^\KDV}(\uu)$ 
to the much larger class 
$C([0,\tau];H^s(\T))$}, 
providing a concrete expression for $\Ld \updl \Ta$.
Moreover, it follows from Lemma \ref{LEM:nonlin2}
that the right-hand side of \eqref{YZ4}
has temporal regularity $1 + \g$
(and, as a result,  so does the left-hand side of~\eqref{YZ4})
which shows an improvement of the temporal regularity
from $\g + \al$ resulting from the application of the sewing lemma (Lemma \ref{LEM:sew}).
As a consequence, 
we see that  
$\I^{X^\KDV}(\uu)$ defined by the right-hand side of \eqref{YZ3}
is still given by the 
 unique limit of 
Riemann-Stieltjes type sums in \eqref{YQ6}.

\begin{remark}\label{REM:num}\rm

We  note that 
the nonlinear Young integral approach in \cite{CG1, CGLLO1, CGLLO2}
allows for a straightforward Euler approximation scheme
to discretize in time (see, for example,  \cite[Subsection~2.3]{CG1};
see also \cite[Subsection 3.3]{GubinelliKdV}).
Similarly, 
by repeating the computation
in \cite[Subsection~2.3]{CG1}
with  \eqref{YZ3} instead of \eqref{YZ2}, 
we obtain an analogous convergence result 
but with a faster convergence rate 
thanks to the higher temporal regularity $1+ \g$
of the second term on the right-hand side of \eqref{YZ3}
(as compared to 
$ - \Ld \updl \Ta$ in \eqref{YZ2}) as pointed above.
Moreover, this approximation scheme \eqref{YZ3} via the normal form approach
does {\it not} assume any positive temporal regularity on a solution,
which is a significant improvement.

We point out that ideas
analogous to the normal form approach 
(and the sewing lemma approach) have been implemented
in time discretization schemes; see, for example, 
\cite{HSch, BSch}.
We also mention 
the work 
\cite{HMSch} on convergence of a time-discretization scheme
for the modulated 
nonlinear Schr\"odinger equation.
For example, 
the approximation via the first order exponential integrator 
for the (unmodulated) KdV 
in \cite{HSch}
is given by 
\begin{align}
 \uu(t_j) 
\approx
 \uu(t_{j+1}) +  X^\KDV_{t_j, t_{j+1}}(\uu (t_{j+1}) ), 
\label{YZ5}
\end{align}

\noi
where $X^\KDV_{t, r}$
is as in 
\eqref{K1} with $w(t) = t$.
(Recall our convention $t_j > t_{j+1}$.)
Compare~\eqref{YZ5} with \eqref{YZ2} and \eqref{YZ3}, 
where the second terms on the right-hand sides of 
\eqref{YZ2} and \eqref{YZ3}
are now viewed as the error terms
in the discretization scheme \eqref{YZ5}, giving (local-in-time) convergence rates.
Note that \eqref{YZ2} and \eqref{YZ3}
hold only in the modulated setting.\footnote{In the unmodulated setting, 
by assuming a higher regularity for a  solution
to the (unmodulated) KdV, 
we can make sense of the second term 
on the right-hand side of \eqref{YZ3}
and carry out local-in-time and global-in-time error analysis;
see \cite{CFO}
for further details.}
In the unmodulated setting, 
we need to use the second order expansion
in \cite{Gub04} via the nonlinear controlled path approach
(and \cite{BIT} via the second normal form reduction)
to replace \eqref{YZ2} (and \eqref{YZ3}, respectively)
which still provides a convergence result.

\end{remark}

\section{Notations and preliminary tools}
\label{SEC:2}

Let $A\les B$ denote an estimate of the form $A\leq CB$ for some constant $C>0$. We write $A\sim B$ if $A\les B$ and $B\les A$, while $A\ll B$ denotes $A\leq c B$ for some small constant $c> 0$. 
We use $C>0$ to denote various constants, which may vary line by line.

In expressing the dependence of a function $u$
on the time variable, we often use the short-hand notation
$u_t = u(t)$,  which is standard in probability theory and stochastic analysis.

We  use $\ft f$ and $\F_x(f)$
to denote
the  Fourier transform
of a function $f$ on $\T$, 
defined by 
\begin{align*}
\ft  f(n)=\int_{\T}f(x)e^{-i n x} d x_\T
= \frac 1 {2\pi}\int_{\T}f(x)e^{-i n x} d x.
\end{align*}

\noi
We impose the mean-zero assumption
on  the equations we consider in this paper
(except for Appendix \ref{SEC:A})
and thus we set  $\Z_* = \Z \setminus \{0\}$.
Given $s \in \R$, 
we define the  Sobolev space $H^{s}(\T)$ 
via the norm: 
\begin{align*}
\|  f  \|^2_{H^{s}(\T)}
& =
\sum_{n \in \Z} \jb{n}^{2s} |\ft f(n)|^2.
  \end{align*}

We often use short-hand notations such as
$C_T H^s_x  = C\big([0, T]; H^s(\M))$, 
 when there is no ambiguity.

Given $k \in \N$
and Hilbert spaces $H_1, H_2$,  
we use 
\begin{align*}
\cL_k(H_1; H_2)
\end{align*}

\noi
 to denote 
the Banach space of bounded $k$-linear operators 
from  $H_1^{\otimes k}$ into $H_2$.
Given a multilinear operator $S$, 
we use the following short-hand notation:
\begin{align}
S(f) = S(f, \dots, f).
\label{short1}
\end{align}

\medskip

We recall the following basic lemma on a  discrete convolution
which follows
from an elementary  computation.
See, for example,  
 \cite[Lemma 4.2]{GTV}. 

\begin{lemma}\label{LEM:SUM}
Let  $\al, \be \in \R$ satisfy
\begin{align*}
 \al \ge \be \ge 0 \qquad \text{and}\qquad  \quad \al+ \be > 1.
\end{align*}

\noi
Then, we have
\begin{align*}
 \sum_{n = n_1 + n_2} \frac{1}{\jb{n_1}^\al \jb{n_2}^\be}
& \les \frac 1{\jb{n}^{ \be - \ld}}
\end{align*}

\noi
for any  $n \in \Z$, 
where $\ld = 
\max( 1- \al, 0)$ when $\al\ne 1$ and $\ld = \eps$ when $\al = 1$ for any $\eps > 0$.

\end{lemma}

\section{Modulated KdV}
\label{SEC:KDV}

We recall the following boundedness 
and nonlinear smoothing property 
of the bilinear operator $X^\KDV_{t, r}$;
see \cite[Proposition 4.1]{CGLLO1}.

\begin{lemma}\label{LEM:nonlin1}
Given $\rho \ge\frac 12$,  $\frac12< \g < 1$, and $T> 0$, 
let  $w$ be $(\rho,\g)$-irregular on $[0, T]$ in the sense of Definition~\ref{DEF:ir}.
Given $0 \le r < t \le T$, let $X^\KDV_{t, r}$
be as in \eqref{K1}.
Suppose that $s \in \R$ satisfies \eqref{nonlin1b}.
Then, 
we have 
$X^\KDV_{t, r} \in \L_2(H^s(\T); H^{s}(\T))$
with the following bound\textup{:}
\begin{align}
 \|X^\KDV_{t, r}\|_{\L_2(H^s; H^{s})}
\les 
(t- r)^\g  \|\Phi^w\|_{  \W^{\rho,\g}_T} 
\label{non1c}
\end{align}

\noi
for any $0 \le r < t \le T$, 
where the implicit constant is independent of $T > 0$.
In addition to~\eqref{nonlin1b}, suppose that $s_0 > s$ satisfies
one of the following conditions\textup{:}
\begin{align}
\begin{split}
\textup{(ii.a)} &\ \  
0 \le s + \rho \le  \tfrac 12
 \quad \text{and} \quad 
s_0 < 2s +  3\rho - \tfrac 32, \\
\textup{(ii.b)} &\ \  
s + \rho >  \tfrac 12 
\quad  \text{and} \quad s_0 \le s + 2\rho - 1.
\end{split}
\label{non1a}
\end{align}

\noi
Then, 
we have 
$X^\KDV_{t, r} \in \L_2(H^s(\T); H^{s_0}(\T))$
with the following bound\textup{:}
\begin{align*}
 \|X^\KDV_{t, r}\|_{\L_2(H^s; H^{s_0})}
\les 
(t-r)^\g  \|\Phi^w\|_{  \W^{\rho,\g}_T}
\end{align*}

\noi
for any $0 \le r < t \le T$.

\end{lemma}

From \eqref{non1a}, we see that, 
given any $s, s_0 \in \R$ with $s_0 \ge s$
(even if $s_0$ is much larger than~$s$), 
the bilinear operator $X^\KDV_{t, r}$
is bounded 
 from $H^s(\T) \times H^s(\T)$ into $H^{s_0}(\T)$, 
 provided that $\rho$ is sufficiently large.

We have the following boundedness property for
the trilinear operator  $\NN^\KDV_{t, r}$
defined in~\eqref{nox1}.

\begin{lemma}\label{LEM:nonlin2}
Given $\rho \ge\frac 12$,  $\frac12< \g < 1$, and $T> 0$, 
let  $w$ be $(\rho,\g)$-irregular on $[0, T]$ in the sense of Definition~\ref{DEF:ir}.
Given $0 \le r < t \le T$, let $\NN^\KDV_{t, r}$
be as in \eqref{nox1}.
Suppose that $s\ge 0$ and $s_0 \in \R$ satisfy
\begin{align}
s + \rho \ge  1
\label{nox1a}
\end{align}

\noi
and 
\begin{align}
s_0 < s + 2\rho - \frac 52.
\label{nox1b}
\end{align}

\noi
Then, 
$\NN^\KDV_{t, r}$  belongs to $\L_3(H^s(\T); H^{s_0}(\T))$, 
satisfying
 the following bound\textup{:}
\begin{align}
 \|\NN^\KDV_{t, r}\|_{\L_3(H^s; H^{s_0})}
\les 
(t- r)^\g  \|\Phi^w\|_{  \W^{\rho,\g}_T}
\label{nox1c}
\end{align}

\noi
for any $0 \le r < t \le T$, 
where the implicit constant is independent of $T > 0$.

\end{lemma}


\begin{proof}

From \eqref{nox1} with \eqref{rho1} and \eqref{K3a}, we have
\begin{align}
 \| \NN^\KDV_{t, r}(f_1, f_2, f_3)\|_{H^{s_0}}
&  \le
(t- r)^\g \|\Phi^w\|_{  \W^{\rho,\g}_T}M^\KDV(f_1, f_2, f_3)
\label{nox2}
\end{align}

\noi
for any $0 \le r < t \le T$, 
where $M^\KDV(f_1, f_2, f_3)$ is given by 
\begin{align}
\begin{split}
& M^\KDV(f_1, f_2, f_3)\\
& \ \  =\bigg\|
\sum_{\substack{n_1, n_2, n_3 \in \Z_*\\n = n_{123}}}
\frac{1}{\jb{n}^{\rho-s_0 - 1}
\jb{n_{12}}^{\rho-1}
\jb{n_1}^{s}\jb{n_2}^{s}\jb{n_3}^{\rho+s}} 
\prod_{j = 1}^3\jb{n_j}^s|\ft f_j(n_j)|
\bigg\|_{\l^2_n}.
\end{split}
\label{nox3}
\end{align}

From \eqref{nox3}, the triangle inequality $\jb{n_{12}}^s \les \jb{n_1}^s\jb{n_2}^s$
for $s \ge 0$, and Cauchy-Schwarz's inequality, we have 
\begin{align}
\begin{split}
 M^\KDV(f_1, f_2, f_3)
&   \le \bigg\|
\sum_{ n_3 \in \Z}
\frac{1}{\jb{n}^{\rho-s_0 - 1}
\jb{n - n_3}^{\rho-1+s}
\jb{n_3}^{\rho+s}} 
\jb{n_3}^s|\ft f_j(n_3)|\bigg\|_{\l^2_n}\\
 & \hphantom{XX} \times 
\sup_{n \in Z^*}\sum_{ n_1 \in \Z}
\ind_{n = n_{123}}
\prod_{j = 1}^2
\jb{n_j}^s|\ft f_j(n_j)|
\\
& 
\le 
 J^\KDV\cdot \prod_{j = 1}^3\|f_j\|_{H^s}, 
\end{split}
\label{nox4}
\end{align}

\noi
where  $ J^\KDV$ is given by 
\begin{align*}
 J^\KDV
= \bigg(\sum_{n,  n_3 \in \Z}
\frac{1}{\jb{n}^{2\rho-2s_0 - 2}
\jb{n - n_3}^{2\rho-2+2s}
\jb{n_3}^{2\rho+2s}} 
\bigg)^\frac 12.
\end{align*}

\noi
By
Lemma \ref{LEM:SUM}
in summing first in $n_3$ and then in $n$, we have $J^\KDV < \infty$, 
\noi
provided that 
$s + \rho  \ge 1$ and $s_0 < s + 2\rho - \frac 52$.
Then, the desired bound \eqref{nox1c} follows
from \eqref{nox2} and \eqref{nox4},  
provided that \eqref{nox1a}
and \eqref{nox1b}
hold.
\end{proof}

\begin{remark}\label{REM:non1}\rm
(i) By considering the contribution from 
the case $|n_1|\sim |n_2|\gg |n|, |n_3|, |n_{12}|$
in \eqref{nox1}, 
it is easy to see that the 
bound \eqref{nox1c} does not hold for $s < 0$.

\smallskip

\noi
(ii) 
By  the so-called Cauchy-Schwarz's argument
(see the proof of \cite[Proposition 4.1]{CGLLO1}), 
\begin{align*}
 \|\NN^\KDV_{t, r}(f_1, f_2, f_3)\|_{H^{s_0}}
 \le
(t-r)^\g  \|\Phi^w\|_{  \W^{\rho,\g}_T} 
\prod_{j = 1}^3\|f_j\|_{H^s}
\cdot \min_{j = 1, \dots 4}\sup_{n_j \in \Z_*} J_{n_j}, 
\end{align*}

\noi
where (by setting $n_4 = -n$) $J_{n_j}$ is given by 
\begin{align*}
J_{n_j} 
& = \bigg(
\sum_{\substack{n_1, n_2, n_3, n_4 \in \Z_*\\ n_{1234}= 0}}
\frac{\ind_{n_{12}\ne 0}}{\jb{n_4}^{2(\rho-s_0 - 1)}
\jb{n_{12}}^{2(\rho-1)}
\jb{n_1}^{2s}\jb{n_2}^{2s}\jb{n_3}^{2(\rho+s)}} 
 \bigg)^\frac{1}{2}.
\end{align*}

\noi
Then, we  bound $J_{n_j}$, using Lemma \ref{LEM:SUM}.
When $j = 3, 4$, 
this argument yields worse conditions than \eqref{nox1a} and \eqref{nox1b}.
When $j = 1, 2$, 
this argument also provides the conditions~\eqref{nox1a} and~\eqref{nox1b}.

\end{remark}

We conclude this section by  presenting a proof of Theorem \ref{THM:UU1}.

\begin{proof}[Proof of Theorem \ref{THM:UU1}]
(i) Let $\G(u) = \G_{u_0}(u)$ be the right-hand side of \eqref{NF4a}.
Under the hypothesis of Theorem~\ref{THM:UU1}\,(i), 
it follows from 
 Lemmas \ref{LEM:nonlin1} and 
\ref{LEM:nonlin2}
that 
\begin{align*}
& \| \G(u) \|_{C_\tau H^s_x}
 \le \| u_0\|_{H^s}
+ C\tau^\g
  \|\Phi^w\|_{  \W^{\rho,\g}_T} 
  \Big(\|u\|_{C_\tau H^s_x}^2 + \|u\|_{C_\tau H^s_x}^3\Big),\\
& \| \G(u)  - \G(v) \|_{C_\tau  H^s_x}\\
& \quad \le 
 C\tau^\g
  \|\Phi^w\|_{  \W^{\rho,\g}_T} 
\Big(1+ \|u\|_{C_\tau H^s_x}^2 + \|v\|_{C_\tau H^s_x}^2\Big) 
\|u- v \|_{C_\tau H^s_x}.
\end{align*}

\noi
for any $0 < \tau \le \min(T, 1)$.
Then, 
 a standard contraction argument
yields unconditional  local well-posedness
of \eqref{NF4a}
with a unique solution $u \in C([0, \tau]; H^s(\T))$, 
where the local existence time 
$\tau > 0$
depends only on 
$\|u_0\|_{H^s}$ and   $\|\Phi^w\|_{  \W^{\rho,\g}_T} $.
Here, the uniqueness of the solution a priori holds only in the ball 
$B_R \subset C([0, \tau]; H^s(\T))$
of radius $R \sim \|u_0\|_{H^s}$ centered at the origin
but, by a standard continuity argument, 
we can extend the uniqueness to the entire space $C([0, \tau]; H^s(\T))$
(by possibly making $\tau$ smaller by a multiplicative constant).
Moreover, by a standard persistence-of-regularity argument, 
we can easily show that the local existence time $\tau$
in fact depends only on 
$\|u_0\|_{L^2}$ and   $\|\Phi^w\|_{  \W^{\rho,\g}_T} $.
Since the argument is standard, we omit details.\footnote{See \cite[Proposition 4.1\,(i)]{CGLLO1}
for the persistence-of-regularity property for the bilinear operator $X_{t, 0}$ defined in \eqref{K1}.
For the trilinear operator $\NN_{t, r}$ defined in \eqref{nox1}, 
we can simply repeat the proof of Lemma \ref{LEM:nonlin2}
with $s_0 = s = 0$
by inserting $\jb{n}^s \les \jb{n_1}^s + \jb{n_2}^s+ \jb{n_3}^s$.
}

Next, we show that the solution $u$ to the normal form equation \eqref{NF4a}
constructed above 
is indeed a solution to the original modulated KdV \eqref{kdv1}.
From the contraction argument above, 
we also have the following Lipschitz bound:
\begin{align}
\| u^{(1)} - u^{(2)}\|_{C_\tau H^s_x}
\les \| u_0^{(1)} - u_0^{(2)}\|_{H^s}
\label{NF4c}
\end{align}

\noi
for any solution $u^{(j)}$  to the normal form equation \eqref{NF4a}
with initial data $u_0^{(j)}$, $j = 1, 2$, 
constructed by the contraction argument described above, 
where $\tau > 0$ denotes the minimum of the 
local existence time for $u^{(j)}$, $j = 1, 2$.
Now, given $u_0 \in H^s(\T)$ and
 a $(\rho,\g)$-irregular function $w$ on $[0, T]$, 
let $u \in C([0, \tau]; H^s(\T))$ be the solution  to \eqref{NF4a}
with $u|_{t = 0} = u_0$, 
and 
$v \in C^\g([0, T]; H^s(\T))$ be the (global-in-time) solution to the 
modulated KdV \eqref{kdv1} with $v|_{t = 0} = u_0$ constructed in \cite{CGLLO1}.
By repeating the computations presented in Subsection~\ref{SUBSEC:NF}
(see also the discussion below for the justification
of the formal steps in~\eqref{NF2}
for functions in $C([0, T]; L^2(\T))$), 
we see that $v$ also satisfies the normal form equation~\eqref{NF4a}
with initial data $u_0$.
Therefore, from \eqref{NF4c}, we conclude that $v \equiv u$ on the time interval $[0, \tau]$
and thus that $u$ satisfies the original modulated KdV \eqref{kdv1}
(more precisely, the Duhamel formulation~\eqref{mild1x}).

It remains to  justify the formal steps in \eqref{NF2}.
For each fixed  $n \in \Z_*$, we need to justify
the following three steps:

\begin{itemize}

\smallskip

\item[(a)] 
switching of the time integration
and the summation at the second equality in \eqref{NF2},

\smallskip

\item[(b)] application of the product rule
on $\ft \uu(t', n_1)\ft \uu(t', n_2)$ 
at the third equality 
in \eqref{NF2},

\smallskip

\item[(c)] 
integration by parts
at the third equality 
in~\eqref{NF2}.

\end{itemize}

\smallskip

\noi
Fix  $u \in C([0, \tau]; L^2(\T))$ for some $0 < \tau \le T$.
(a)~Since we have
\begin{align*}
\sup_{0 \le t' \le \tau} \sum_{\substack{n_1, n_2 \in \Z_*\\n = n_1 + n_2}}
|e^{i  \Xi_\KDV (\bar n)w(t')} \ft \uu(t', n_1)\ft \uu(t', n_2) |
\le \| \uu \|_{C_\tau L^2_x}^2
\end{align*}

\noi
for each fixed  $n \in \Z_*$,
the second equality of \eqref{NF2}
follows from 
the Fubini  theorem.
(b)~From~\eqref{kdv1b}, we have 
\begin{align*}
|\dt \ft \uu (t, n)| \le |n| \| u (t) \|_{L^2_x}^2
\end{align*}

\noi
for any $0 \le t \le T$.
In particular, 
for each fixed  $n \in \Z_*$,
we have $\ft \uu (\cdot, n)$
is a $C^1$-function in $t$, 
which justifies the application of 
the product rule 
at the third step in \eqref{NF2}.
(c)~Furthermore, by noting from~\eqref{rho2} that 
$\Phi^w_{t_1, t_2}(\Xi_\KDV (\bar n))$
is a $C^1$-function in $t_2$, 
the 
integration by parts
at the third equality 
in~\eqref{NF2}
is justified.
The discussion above shows that 
any solution $u \in C([0, \tau]; L^2(\T))$ to 
the Duhamel formulation \eqref{mild1x} of the modulated KdV \eqref{kdv1}
with $u|_{t = 0} = u_0$ for some $0 < \tau \le T$
(which does not have to be the solution constructed above
or that constructed in \cite{CGLLO1})
is also a solution the normal form equation~\eqref{NF4a}, 
which proves unconditional uniqueness
in the entire space $C([0, \tau]; L^2(\T))$
without intersecting with any auxiliary space.

Lastly, suppose in addition that 
 $w$ is  $(\rho,\g)$-irregular  on $\R_+$.
Then, the claimed unconditional global well-posedness of
the modulated KdV \eqref{kdv1}
in $H^s(\T)$ follows from 
the global well-posedness proven in \cite{CGLLO1}
and the local-in-time unconditional uniqueness
discussed above which can be iteratively applied
to yield global-in-time unconditional uniqueness.

\medskip

\noi
(ii) 
The claimed nonlinear smoothing follows
from applying
 Lemmas \ref{LEM:nonlin1} and 
\ref{LEM:nonlin2}
to the normal form equation \eqref{NF4a}. 
\end{proof}

\section{Modulated BO}
\label{SEC:BO}

We present a proof of Theorem \ref{THM:UU2}
for  the modulated BO \eqref{BO} on $\T$.
The bilinear driver associated with 
the modulated BO~\eqref{BO} on $\T$ is given by 
\begin{align}
X^\BO_{t,r}( f _1, f _2)
=\int_r^t  \uw  (t')^{-1}
\dx \big( (\uw (t')   f _1)( \uw (t')  f_2)\big) dt',
\label{BX1}
\end{align}

\noi
where 
\begin{align*}
\uw (t)=e^{w(t) \H\dx^2 }
\end{align*}

\noi
denotes 
the modulated linear propagator
for \eqref{BO}.
By taking the Fourier transform, we have 
\begin{align*}
\F\big( X^\BO_{t,r}  ( f _1, f _2) \big)(n)  =
in  \sum_{ \sub {n_1, n_2 \in \Z_*\\ n = n_1+n_2}}
 \Phi_{t,r}^{w}(\Xi_\BO (\bar n))\ft  f_1 ( n_1 )  \ft   f_2 (n_2), 
\end{align*}

\noi
where $\Phi_{t,r}^w$ is as in  \eqref{rho2}
and 
$\Xi_\BO (\bar n)$ denotes the resonance function for BO given by 
\begin{align}
\Xi_\BO (\bar n)
= \Xi_\BO(n,n_1,n_2) = - |n| n +  | n_1 |n_1 +  |n_2| n_2.
\label{BX4}
\end{align}

\noi
Given $0 \le r< t  \le T$, define the trilinear operator  $\NN^\BO_{t, r}(f_1, f_2, f_3)$, acting on functions on $\T$, 
 by 
\begin{align}
\begin{split}
& \F_x\big(\NN^\BO_{t, r}(f_1, f_2, f_3)\big)(n)\\
& \quad =   - 
2 n  \sum_{\substack{n_1, n_2, n_3 \in \Z_*\\n = n_{123}}}
 \Phi^w_{t, r}(\Xi_\BO (n, n_{12}, n_3))
e^{i  \Xi_\BO (n_{12}, n_1, n_2)w(r)}n_{12} 
\prod_{j = 1}^3\ft f_j(n_j), 
\end{split} 
\label{BX5}
\end{align}

\noi
where we used the convention \eqref{no1}.
Let $\uu(t) = \uw(t)^{-1} u(t)$ denote the 
modulated interaction representation 
of a solution $u$ to  (the Duhamel formulation of)~\eqref{BO}.
Then, 
proceeding as in Subsection \ref{SUBSEC:NF}, 
we have the following normal form equation for 
the 
modulated BO \eqref{BO}:
\begin{align*}
\ft \uu(t, n) - \ft \uu(0, n) 
 = \F_x\big(X^\BO_{t, 0}(\uu(0))\big)(n)
 +  \int_0^t 
\F_x\big(\NN^\BO_{t, t'}(\uu(t'))\big)(n)dt', 
\end{align*}

\noi
where we used the short-hand notation \eqref{short1}.

We recall the following boundedness 
and nonlinear smoothing property 
of the bilinear operator $X^\BO_{t, r}$;
see \cite[Proposition 5.1]{CGLLO1}.

\begin{lemma}\label{LEM:nonlin3}
Given $\rho \ge 1$,  $\frac12< \g < 1$, and $T> 0$, 
let  $w$ be $(\rho,\g)$-irregular on $[0, T]$ in the sense of Definition~\ref{DEF:ir}.
Given $0 \le r < t \le T$, let $X^\BO_{t, r}$
be as in \eqref{BX1}.
Suppose that $s \in \R$ satisfies~\eqref{regBO1}.
Then, 
we have $X_{t, r}^\BO \in \L_2(H^s(\T); H^{s}(\T))$
with the following bound\textup{:}
\begin{align*}
 \|X^\BO_{t, r}\|_{\L_2(H^s; H^{s})}
\les 
(t-r)^\g  \|\Phi^w\|_{  \W^{\rho,\g}_T}
\end{align*}

\noi
for any $0 \le r < t \le T$, 
where the implicit constant is independent of $T > 0$.
In addition to~\eqref{regBO1}, suppose that 
$\rho  > 1$ and $s_0 > s$ satisfy
one of the following conditions\textup{:}
\begin{align*}
\textup{(ii.a)} &\ \
0 \le s < s_0 \le \rho - 1,\\
\textup{(ii.b)} &\ \
 s <0
\quad 
\text{and}\quad 
s_0 \le s + \rho - 1.
\end{align*}

\noi
Then, 
we have 
$X^\BO_{t, r} \in \L_2(H^s(\T); H^{s_0}(\T))$
with the following bound\textup{:}
\begin{align*}
 \|X^\BO_{t, r}\|_{\L_2(H^s; H^{s_0})}
\les 
(t- r)^\g  \|\Phi^w\|_{  \W^{\rho,\g}_T}
\end{align*}

\noi
for any $0 \le r < t \le T$.

\end{lemma}

Once we prove the following lemma
on the trilinear operator $\NN_{t, r}^\BO$
defined in \eqref{BX5},
Theorem~\ref{THM:UU2} 
for the modulated BO \eqref{BO} follows
from proceeding as in 
the proof of Theorem~\ref{THM:UU1}
for the modulated KdV \eqref{kdv1}
presented in Section \ref{SEC:KDV}.
We omit details of the proof of 
Theorem~\ref{THM:UU2}.

\begin{lemma}\label{LEM:nonlin4}
Given $\rho \ge1$,  $\frac12< \g < 1$, and $T> 0$, 
let  $w$ be $(\rho,\g)$-irregular on $[0, T]$ in the sense of Definition~\ref{DEF:ir}.
Given $0 \le r < t \le T$, let $\NN^\BO_{t, r}$
be as in \eqref{BX5}.
Suppose that $s\ge 0$ and $s_0 \in \R$ satisfy
\begin{align}
s + \rho \ge  1, \qquad 2s + \rho > \frac 32
\label{Bnox1a}
\end{align}

\noi
and 
\begin{align}
s_0 < s + \rho - \frac 52.
\label{Bnox1b}
\end{align}

\noi
Then, 
$\NN_{t, r}^\BO$  belongs to $\L_3(H^s(\T); H^{s_0}(\T))$, 
satisfying
 the following bound\textup{:}
\begin{align}
 \|\NN^\BO_{t, r}\|_{\L_3(H^s; H^{s_0})}
\les 
(t- r)^\g  \|\Phi^w\|_{  \W^{\rho,\g}_T}
\label{Bnox1c}
\end{align}

\noi
for any $0 \le r <  t \le T$, 
where the implicit constant is independent of $T > 0$.

\end{lemma}

Here, the condition $\rho \ge 1$ may be relaxed, 
but, in view of Lemma \ref{LEM:nonlin3}, 
 we keep it as it is since we need both Lemmas \ref{LEM:nonlin3}
 and \ref{LEM:nonlin4} to prove Theorem \ref{THM:UU2}.
A similar comment applies to Lemma \ref{LEM:nonlin6}.

Just as in the modulated KdV case, 
the bound \eqref{Bnox1c}
does not hold for $s < 0$; see
Remark~\ref{REM:non1}.
We conclude this section by  presenting a proof of  Lemma \ref{LEM:nonlin4}.

\begin{proof}[Proof of Lemma \ref{LEM:nonlin4}]

From \eqref{BX5} with \eqref{rho1} and \eqref{BX4}, we have
\begin{align*}
 \| \NN^\BO_{t, r}(f_1, f_2,  f_3)\|_{H^{s_0}}
&  \le
(t-r)^\g \|\Phi^w\|_{  \W^{\rho,\g}_T}M^\BO(f_1, f_2, f_3)
\end{align*}

\noi
for any $0 \le r < t \le T$, 
where $M^\BO(f_1, f_2, f_3)$ is given by 
\begin{align*}
& M^\BO(f_1, f_2, f_3)\\
& \ \  =\bigg\|
\sum_{\substack{n_1, n_2, n_3 \in \Z_*\\n = n_{123}}}
\frac{\jb{\Xi_\BO(n, n_{12}, n_3)}^{-\rho}}{\jb{n}^{-s_0 - 1}
\jb{n_{12}}^{-1}
\jb{n_1}^{s}\jb{n_2}^{s}\jb{n_3}^{s}} 
\prod_{j = 1}^3\jb{n_j}^s|\ft f_j(n_j)|
\bigg\|_{\l^2_n}.
\end{align*}

\noi
Then, arguing as in \eqref{nox4}, 
it suffices to show that 
\begin{align*}
 J^\BO
= \bigg(\sum_{n,  n_3 \in \Z_*}
\frac{\ind_{n = n_{123}}
\ind_{n_{12} \ne0}
\cdot \jb{\Xi_\BO(n, n_{12}, n_3)}^{-2\rho}}
{\jb{n}^{-2s_0 - 2}
\jb{n - n_3}^{-2+2s}
\jb{n_3}^{2s}} 
\bigg)^\frac 12 < \infty.
\end{align*}

From \eqref{BX4}, we have, under $n = n_{12} + n_3$,  

\smallskip

\begin{itemize}
\item[(a)]
If $n_{12}n_3>0$ (i.e.~$n$, $n_{12}$, and $n_3$ all have the same sign), 
then  $|\Xi_\BO(n, n_{12}, n_3)|=2|n_{12}n_3|$.

\smallskip

\item[(b)]
If $n, n_{12}>0>n_3$,  then  $|\Xi_\BO(n, n_{12}, n_3)|=2|n n_3|$.

\smallskip

\item[(c)]

If $n, n_3>0>n_{12}$,  then  $|\Xi_\BO(n, n_{12}, n_3)|=2|n n_{12}|$.

\smallskip

\item[(d)]
If $n_{12}>0>n, n_3$,  then  $|\Xi_\BO(n, n_{12}, n_3)|=2|n n_{12}|$.

\smallskip

\item[(e)]
If $n_3>0>n, n_{12}$,  then  $|\Xi_\BO(n, n_{12}, n_3)|=2|n n_3|$.

\end{itemize}

\smallskip

\noi
By symmetry, 
we only consider Cases (a), (b),  and (c) in the following.

\medskip

\noi
$\bul$ {\bf Case (a):} $n_{12} n_3 >0$.\\
\indent
In this case, we have
\begin{align*}
( J^\BO)^2
\les  \sum_{n,  n_3 \in \Z}
\frac{1}
{\jb{n}^{-2s_0 - 2}
\jb{n - n_3}^{2\rho-2+2s}
\jb{n_3}^{2\rho+2s}} .
\end{align*}

\noi
By
Lemma \ref{LEM:SUM}
in summing first in $n_3$ and then in $n$, we have $J^\BO < \infty$, 
provided that 
$s + \rho  \ge 1$ and $s_0 < s + \rho - \frac 52$.

\medskip

\noi
$\bul$ {\bf Case (b):} $n, n_{12}>0>n_3$.\\
\indent
In this case, we have $|n_{12}| > |n_3|$ and thus $|n_{12}| \ges |n|$,
since $n = n_{12} + n_3$.
Then,  we have
\begin{align*}
 (J^\BO)^2
& \les  
\sum_{n,  n_3 \in \Z}
\frac{\ind_{n = n_{123}}\ind_{|n_{12}|\sim |n_3| \ges |n|}}
{\jb{n}^{2\rho-2s_0 - 2}
\jb{n_3}^{2\rho+4s - 2}} \\
& \quad + \sum_{n,  n_3 \in \Z}
\frac{\ind_{n = n_{123}}\ind_{|n_{12}|\sim |n| \ges |n_3|}}
{\jb{n}^{2\rho-2s_0 +2s- 4}
\jb{n_3}^{2\rho+2s}} 
< \infty, 
\end{align*}

\noi
provided that 
$2s + \rho  > \frac 32$ and $s_0 < s + \rho - \frac 52$.

\medskip

\noi
$\bul$ {\bf Case (c):} $n, n_3>0>n_{12}$.\\
\indent
In this case, we have $|n_{3}| > |n_{12}|$ and thus $|n_{3}| \ges |n|$,
since $n = n_{12} + n_3$.
Then, by Lemma~\ref{LEM:SUM},  we have
\begin{align*}
( J^\BO)^2
& \les  \sum_{n,  n_3 \in \Z}
\frac{\ind_{n = n_{123}}\ind_{|n_{12}|\sim |n_3| \ges |n|}}
{\jb{n}^{2\rho-2s_0 - 2}
\jb{n_3}^{2\rho + 4s-2}} \\
& 
\quad +   \sum_{n,  n_3 \in \Z}
\frac{\ind_{n = n_{123}}\ind_{|n|\sim |n_3| \ges |n_{12}|}}
{\jb{n}^{2\rho-2s_0 +2s - 2}
\jb{n - n_3}^{2\rho-2+2s}} < \infty, 
\end{align*}

\noi
provided that 
$2s + \rho  > \frac 32$ and $s_0 < \min\big(2s + 2\rho -3, s + \rho - \frac 32\big)$.

\medskip

Therefore, 
putting all the cases together, 
we conclude that  the desired bound \eqref{Bnox1c} holds, 
provided that \eqref{Bnox1a}
and \eqref{Bnox1b}
hold.
\end{proof}

\section{Modulated dNLS}
\label{SEC:NLS}

In this section, we briefly discuss a proof of Theorem \ref{THM:UU3}
for the modulated dNLS \eqref{dNLS1} on $\T$.
The bilinear driver associated with 
the modulated dNLS~\eqref{dNLS1} on $\T$ is given by 
\begin{align}
X^\dNLS_{t,r}( f _1, f _2)
=-i \int_r^t  \uw  (t')^{-1}
\dx \big( (\uw (t')   f _1)( \uw (t')  f_2)\big) dt',
\label{SX1}
\end{align}

\noi
where 
\begin{align*}
\uw (t)=e^{i w(t) \dx^2 }
\end{align*}

\noi
denotes 
the modulated linear propagator
for \eqref{dNLS1}.
By taking the Fourier transform, we have 
\begin{align*}
\F\big( X^\dNLS_{t,r}  ( f _1, f _2) \big)(n)  =
n  \sum_{ \sub {n_1, n_2 \in \Z_*\\ n = n_1+n_2}}
 \Phi_{t,r}^{w}(\Xi_\dNLS (\bar n))\ft  f_1 ( n_1 )  \ft   f_2 (n_2), 
\end{align*}

\noi
where $\Phi_{t,r}^w$ is as in  \eqref{rho2}
and 
$\Xi_\dNLS (\bar n)$ denotes the resonance function for dNLS given by 
\begin{align}
\begin{split}
\Xi_\dNLS (\bar n)
& = \Xi_\dNLS(n,n_1,n_2) = n^2 - n_1^2 - n_2^2\\
& =  2 n_1 n_2, 
\end{split}
\label{SX4a}
\end{align}

\noi
where the last equality holds under $n = n_1+ n_2$.
Given $0 \le r < t \le T$, 
define the trilinear operator $\NN^\dNLS_{t, r}(f_1, f_2, f_3)$, acting on functions on $\T$, 
 by 
\begin{align}
\begin{split}
& \F_x\big(\NN^\dNLS_{t, r}(f_1, f_2, f_3)\big)(n)\\
& \quad =   
- 2 i  n   \sum_{\substack{n_1, n_2, n_3 \in \Z_*\\n = n_{123}}}
 \Phi^w_{t, r}(\Xi_\dNLS (n, n_{12}, n_3))
e^{i  \Xi_\dNLS (n_{12}, n_1, n_2)w(r)}n_{12} 
\prod_{j = 1}^3\ft f_j(n_j), 
\end{split} 
\label{SX5}
\end{align}

\noi
where we used the convention \eqref{no1}.
Let $\uu(t) = \uw(t)^{-1} u(t)$ denote the 
modulated interaction representation 
of a solution $u$ to (the Duhamel formulation of)~\eqref{dNLS1}.
Then, 
proceeding as in Section \ref{SEC:KDV}, 
we have the following normal form equation for 
the 
modulated dNLS \eqref{dNLS1}:
\begin{align*}
\ft \uu(t, n) - \ft \uu(0, n) 
 = \F_x\big(X^\dNLS_{t, 0}(\uu(0))\big)(n)
 +  \int_0^t 
\F_x\big(\NN^\dNLS_{t, t'}(\uu(t'))\big)(n)dt', 
\end{align*}

\noi
where we used the short-hand notation \eqref{short1}.

The next lemma on the boundedness 
and nonlinear smoothing property 
of the bilinear operator $X^\dNLS_{t, r}$
follows from Cases (i.a) and (iii.a)
in the proof of 
\cite[Proposition 5.1]{CGLLO1}
for the bilinear driver $X^\BO_{t, r}$ 
for the modulated BO defined in~\eqref{BX1}.

\begin{lemma}\label{LEM:nonlin5}
Given $\rho \ge 1$,  $\frac12< \g < 1$, and $T> 0$, 
let  $w$ be $(\rho,\g)$-irregular on $[0, T]$ in the sense of Definition~\ref{DEF:ir}.
Given $0 \le r < t \le T$, let $X^\dNLS_{t, r}$
be as in \eqref{SX1}.
Suppose that $s \in \R$ satisfies 
\begin{align}
 s+ \rho \ge 0 \qquad\text{and}\qquad 
s >  \frac 32 - 2\rho.
\label{regS1}
\end{align}

\noi
Then, 
we have $X_{t, r}^\dNLS \in \L_2(H^s(\T); H^{s}(\T))$
with the following bound\textup{:}
\begin{align*}
 \|X^\dNLS_{t, r}\|_{\L_2(H^s; H^{s})}
\les 
(t- r)^\g  \|\Phi^w\|_{  \W^{\rho,\g}_T}
\end{align*}

\noi
for any $0 \le r < t \le T$, 
where the implicit constant is independent of $T > 0$.
In addition to~\eqref{regS1}, suppose that 
$\rho  > 1$ and $s_0 > s$ satisfy
\begin{align*}
s+ \rho \ge 0, \qquad s_0 \le s + \rho - 1, \qquad \text{and}\qquad 
s_0 < 2s + 2\rho - \frac 32.
\end{align*}

\noi
Then, 
we have 
$X^\dNLS_{t, r} \in \L_2(H^s(\T); H^{s_0}(\T))$
with the following bound\textup{:}
\begin{align*}
 \|X^\dNLS_{t, r}\|_{\L_2(H^s; H^{s_0})}
\les 
(t- r)^\g  \|\Phi^w\|_{  \W^{\rho,\g}_T}
\end{align*}

\noi
for any $0 \le r< t \le T$.

\end{lemma}

Once we prove the following lemma
on the trilinear operator $\NN_{t, r}^\dNLS$
defined in \eqref{SX5}, 
Theorem~\ref{THM:UU3} 
for the modulated dNLS \eqref{dNLS1} follows
from proceeding as in 
the proof of Theorem~\ref{THM:UU1}
for the modulated KdV~\eqref{kdv1}
presented in Section \ref{SEC:KDV}.
We omit details of the proof of 
Theorem~\ref{THM:UU3}.

\begin{lemma}\label{LEM:nonlin6}
Given $\rho \ge 1$,  $\frac12< \g < 1$, and $T> 0$, 
let  $w$ be $(\rho,\g)$-irregular on $[0, T]$ in the sense of Definition~\ref{DEF:ir}.
Given $0 \le r < t \le T$, let $\NN^\dNLS_{t, r}$
be as in \eqref{SX5}.
Suppose that $s\ge 0$ and $s_0 \in \R$ satisfy
\begin{align}
s + \rho \ge  1 
\label{Snox1a}
\end{align}

\noi
and 
\begin{align}
s_0 < s + \rho - \frac 52.
\label{Snox1b}
\end{align}

\noi
Then, 
$\NN_{t, r}^\dNLS$  belongs to $\L_3(H^s(\T); H^{s_0}(\T))$, 
satisfying
 the following bound\textup{:}
\begin{align}
 \|\NN^\dNLS_{t, r}\|_{\L_3(H^s; H^{s_0})}
\les 
(t- r)^\g  \|\Phi^w\|_{  \W^{\rho,\g}_T}
\label{Snox1c}
\end{align}

\noi
for any $0 \le r < t \le T$, 
where the implicit constant is independent of $T > 0$.

\end{lemma}

Note that 
the bound \eqref{Snox1c}
does not hold for $s < 0$; see
Remark~\ref{REM:non1}.
We conclude this section by  presenting a proof of  Lemma \ref{LEM:nonlin6}.

\begin{proof}[Proof of Lemma \ref{LEM:nonlin6}]

From \eqref{SX5} with \eqref{rho1} and \eqref{BX4}, we have
\begin{align*}
 \| \NN^\dNLS_{t, r}(f_1, f_2, f_3)\|_{H^{s_0}}
&  \le
(t-r)^\g \|\Phi^w\|_{  \W^{\rho,\g}_T}M^\dNLS(f_1, f_2, f_3)
\end{align*}

\noi
for any $0 \le r < t \le T$, 
where $M^\dNLS(f_1, f_2, f_3)$ is given by 
\begin{align*}
& M^\dNLS(f_1, f_2, f_3)\\
& \ \  =\bigg\|
\sum_{\substack{n_1, n_2, n_3 \in \Z_*\\n = n_{123}}}
\frac{\jb{\Xi_\dNLS(n, n_{12}, n_3)}^{-\rho}}{\jb{n}^{-s_0 - 1}
\jb{n_{12}}^{-1}
\jb{n_1}^{s}\jb{n_2}^{s}\jb{n_3}^{s}} 
\prod_{j = 1}^3\jb{n_j}^s|\ft f_j(n_j)|
\bigg\|_{\l^2_n}.
\end{align*}

\noi
Then, arguing as in \eqref{nox4}, 
it suffices to show that 
\begin{align}
 J^\dNLS
= \bigg(\sum_{n,  n_3 \in \Z_*}
\frac{\ind_{n = n_{123}}
\ind_{n_{12} \ne0}
\cdot \jb{\Xi_\dNLS(n, n_{12}, n_3)}^{-2\rho}}
{\jb{n}^{-2s_0 - 2}
\jb{n - n_3}^{-2+2s}
\jb{n_3}^{2s}} 
\bigg)^\frac 12 < \infty.
\label{Snox5}
\end{align}

From \eqref{SX4a}, we have
\begin{align}
\Xi_\dNLS(n, n_{12}, n_3)
=  2 n_{12} n_3.
\label{SX4b}
\end{align}

\noi
 under $n = n_{12} + n_3$.
Then, from \eqref{Snox5}
and \eqref{SX4b}, we have
\begin{align*}
( J^\dNLS)^2
\les  \sum_{n,  n_3 \in \Z}
\frac{1}
{\jb{n}^{-2s_0 - 2}
\jb{n - n_3}^{2\rho-2+2s}
\jb{n_3}^{2\rho+2s}} .
\end{align*}

\noi
By
Lemma \ref{LEM:SUM}
in summing first in $n_3$ and then in $n$, we have $J^\dNLS < \infty$, 
provided that \eqref{Snox1a}
and \eqref{Snox1b} hold.
This proves 
\eqref{Snox1c}.
\end{proof}

\appendix

\section{Unconditional uniqueness for the  modulated cubic NLS} 
\label{SEC:A}

In this appendix, we study  the 
 following modulated cubic NLS on $\T$:
\begin{align}
i \dt u + \dx^2 u \cdot \dt w= |u|^2u.
\label{xNLS1}
\end{align}

\noi
In \cite{CG1},
Chouk and the first author showed that 
if   $w$ is $(\rho,\g)$-irregular on $\R_+$
for some $\rho > \frac 12 $ (and $\frac 12 < \g < 1$), 
then \eqref{xNLS1} is  globally well-posed
in $L^2(\T)$; see \cite[Theorem~1.8]{CG1}.

Our main goal in this appendix is to 
implement a Poincar\'e-Dulac normal form reduction for \eqref{xNLS1}, 
by splitting the nonlinearity in the non-resonant and resonant parts as in~\eqref{xNLS2} below, 
and prove the following theorem.

\begin{theorem}\label{THM:UU4}
Given $\rho > \frac 23$,  $\frac12< \g < 1$, and $T> 0$, 
let  $w$ be $(\rho,\g)$-irregular on $[0, T]$ in the sense of Definition~\ref{DEF:ir}.
Then, 
given any    $s\ge \frac 16$, 
the modulated cubic NLS equation~\eqref{xNLS1}
on  $\T$
is unconditionally locally well-posed in $H^s(\T)$.
In particular, 
if, in addition, $w$ is  $(\rho,\g)$-irregular  on $\R_+$, 
then the modulated cubic NLS equation \eqref{xNLS1}
on  $\T$
is unconditionally globally well-posed in $H^s(\T)$.

\end{theorem}

\begin{remark}\label{REM:sharp2} \rm

In view of  the cubic nonlinearity in \eqref{xNLS1}
and the (essentially scaling invariant) embedding $H^\frac 16 (\T)\subset L^3(\T)$, 
the regularity threshold $s \ge \frac 16$ 
in Theorem \ref{THM:UU4} 
is sharp
(within the $L^2$-based Sobolev scale). 
\end{remark}


\medskip

The remaining part of this appendix is devoted to a proof of Theorem \ref{THM:UU4}.
As pointed out in Remark \ref{REM:NLS}, 
the  modulated cubic NLS \eqref{xNLS1} 
is {\it not} non-resonant, 
contrary to 
 the non-resonant models studied in 
Sections~\ref{SEC:KDV}, 
\ref{SEC:BO}, and \ref{SEC:NLS}.
Our strategy is the following.
We first split the nonlinearity into 
the non-resonant and resonant parts, 
and then apply a normal form reduction
only to the non-resonant part, 
while the  resonant part satisfies a trivial bound;
see Lemma \ref{LEM:nonlin7}.
See also Lemma \ref{LEM:nonlin9}\,(i).
Here, the resonant
part
refers to the nonlinear interaction coming
from the case 
$\Xi_\NLS (\bar n) =0$, namely $n = n_1$ or $n_3$,  
where 
$\Xi_\NLS (\bar n)$ is the resonance function
for the cubic NLS defined in 
\eqref{Xi9}.
See Remark \ref{REM:NLS1} for an alternative approach.

Write \eqref{xNLS1} as 
\begin{align}
\begin{split}
i \dt u + \dx^2 u \cdot \dt w
& = \NN(u) + \RR(u), 
\end{split}
\label{xNLS2}
\end{align}

\noi
where $\NN(u)$ and $\RR(u)$ denote the non-resonant 
 and    resonant parts of the nonlinearity,
respectively, 
defined by 
\begin{align}
\begin{split}
\F_x\big(\NN(u_1,u_2,u_3)\big)(n) 
& = \sum_{\substack{n = n_{123}^*\\ n\neq n_{1},n_{3}}}
\ft u_1 (n_1) \cj{\ft u_2 (n_2)} \ft u_3 (n_3), \\
\F_x\big(\RR(u_1,u_2,u_3)\big) (n) 
& = 
 \sum_{n_1 \in \Z}
\ft u_1 (n_1) \cj{\ft u_2 (n_1)} \ft u_3 (n)\\
& \quad + 
 \sum_{n_3 \in \Z}
\ft u_1 (n) \cj{\ft u_2 (n_3)} \ft u_3 (n_3)\\
& \quad 
- 
\ft u_1 (n) \cj{\ft u_2 (n)} \ft u_3 (n).
\end{split}
\label{non1}
\end{align}

\noi
Here,    we used 
 the short-hand notation \eqref{short1}
and 
the following short-hand notation:
\begin{align}
n_{1\cdots k}^* = n_1 - n_2 + n_3 -  \cdots + n_k
\label{no2}
\end{align}

\noi
for $k \in 2\N + 1$.
In the following, we also view $\RR$ 
as a trilinear operator acting on 
 functions on $\T$.

Let $\uu(t) = \uw(t)^{-1} u(t)$ denote the 
modulated interaction representation 
of a solution $u$ to~\eqref{xNLS2}.
Then, from \eqref{non1}, we have 
\begin{align}
\begin{split}
\dt \ft \uu(t, n)
& = - i \F_x\big(\uw  (t')^{-1}\NN(u)\big)(t, n)
- i \F_x\big(\uw  (t')^{-1}\RR(u)\big)(t, n)
\\
& = 
-i  
\sum_{\substack{n = n_{123}^*\\ n\neq n_{1},n_{3}}}
e^{i  \Xi_\NLS (\bar n)w(t)} 
\prod_{j = 1}^3 {\ft \uu}^*(t, n_j)\\
& \quad  
- 2i \sum_{k \in \Z} 
    |\ft \uu(t, k)|^2\ft \uu(t, n)
+   i     |\ft \uu(t, n)|^2\ft \uu(t, n)
\end{split}
\label{xNLS2a}
\end{align}

\noi
for $n \in \Z$, 
where $\Xi_\NLS (\bar n)$ denotes the resonance function for
\eqref{xNLS2} defined in \eqref{Xi9}
and 	
\begin{align}
\ft \uu^*(t, n_j)
= \begin{cases}
\ft \uu(t, n_j), & \text{when $j$ is odd,}\rule[-3mm]{0pt}{0pt} \\
\cj{\ft \uu(t, n_j)}, & \text{when $j$ is even.}
\end{cases}
\label{gauge2}
\end{align}

\noi
Then,  we have 
\begin{align}
\begin{split}
\ft \uu(t, n) - \ft \uu(0, n) 
& = 
- i \int_0^t \F_x\big(\uw  (t')^{-1}\NN(\uw  (t')\uu(t'))\big) dt'\\
& \quad  - i \int_0^t \F_x\big(\uw  (t')^{-1}\RR(\uw  (t')\uu(t'))\big) dt'\\
&
= 
-i  
\int_0^t
\sum_{\substack{n = n_{123}^*\\ n\neq n_{1},n_{3}}}
e^{i  \Xi_\NLS (\bar n)w(t')} 
\prod_{j = 1}^3 \ft \uu^*(t', n_j)dt'\\
& \quad
- 2i  \int_0^t \sum_{k \in \Z} 
    |\ft \uu(t', k)|^2\ft \uu(t', n)dt'\\
& \quad  +   i   \int_0^t  |\ft \uu(t', n)|^2\ft \uu(t', n) dt'.
\end{split}
\label{xNLS3}
\end{align}

\noi
The first  term on the right-hand side of \eqref{xNLS3} represents 
the non-resonant   contribution, 
while the second and third terms represent the resonant contribution.
In the following, we apply a normal form reduction
to the non-resonant part, 
corresponding to 
a Poincar\'e-Dulac normal form reduction;
see Remark~\ref{REM:NF1}.

The trilinear  driver associated with
the non-resonant part in \eqref{xNLS3}  is given by 
\begin{align}
X^\NLS_{t,r}( f _1, f _2, f_3)
=-i \int_r^t  \uw  (t')^{-1}
\NN \big( \uw (t')   f _1, \uw (t')  f_2, 
 \uw (t')  f_3
 \big) dt',
\label{xX1}
\end{align}

\noi
where 
\begin{align*}
\uw (t)=e^{i w(t) \dx^2 }
\end{align*}

\noi
denotes 
the modulated linear propagator
for \eqref{xNLS2}.
By taking the Fourier transform, 
it follows from \eqref{non1} that  
\begin{align}
\F\big( X^\NLS_{t,r}  ( f _1, f _2, f_3) \big)(n)  =
-i  \sum_{ \sub {n = n_{123}^*\\n \ne n_1, n_3}}
 \Phi_{t,r}^{w}(\Xi_\NLS (\bar n))\ft  f_1 ( n_1 ) \cj{ \ft   f_2 (n_2)}
 \ft  f_3 ( n_3 ), 
\label{xX2}
\end{align}

\noi
where $\Phi_{t,r}^w$ 
and 
$\Xi_\NLS (\bar n)$ 
are as in  \eqref{rho2}
and  \eqref{Xi9}.

Proceeding as in \eqref{NF2} with \eqref{xNLS3} and \eqref{NS1}, we have 
\begin{align}
\begin{split}
- i  & \int_0^t \F_x\big(\uw(t')^{-1}\NN(\uw(t') \uu(t')\big) dt'\\
& =  i  
\sum_{\substack{n = n_{123}^*\\ n\neq n_{1},n_{3}}}
 \int_0^t
 \partial_{t'} \Phi^w_{t, t'}(\Xi_\NLS (\bar n))
 \prod_{j = 1}^3 \ft \uu^*(t', n_j)
 dt'\\
& = \F_x\big(X^\NLS_{t, 0}(\uu(0), \uu(0), \uu(0))\big)(n)  \\
&  \quad -  i \sum_{\substack{n = n_{123}^*\\ n\neq n_{1},n_{3}}}
 \int_0^t
 \Phi^w_{t, t'}(\Xi_\NLS (\bar n))
\, \dt \bigg( \prod_{j = 1}^3 \ft \uu^*(t', n_j)\bigg) 
 dt'.
 \end{split}
\label{xNLS4}
\end{align}

\noi
By applying the product rule to the last term in \eqref{xNLS4}
and substituting
\eqref{xNLS2a}, 
we obtain two 
quintic terms, 
coming from the non-resonant and resonant terms
in \eqref{xNLS2a}, respectively.
In order to express these terms, 
we introduce two  
quintilinear operators  
$\NN \NN^\NLS_{t, r}(f_1, \dots, f_5)$
and 
$\NN \RR^\NLS_{t, r}(f_1, \dots, f_5)$, acting on functions on $\T$.
The first operator
$\NN \NN^\NLS_{t, r}(f_1, \dots, f_5)$
 represents the contribution
coming from the non-resonant part in substituting \eqref{xNLS2a}
and is defined by 
\begin{align}
\begin{split}
 \F_x & \big(\NN \NN^\NLS_{t, r}(f_1, \dots, f_5)\big)(n)\\
&  =   - 
2   \sum_{\substack{n = n_{12345}^*\\
n \ne n_{123}^*, n_5\\
n_{123}^* \ne n_1, n_3
}}
 \Phi^w_{t, r}(\Xi_\NLS (n, n_{123}^*, n_4, n_5))
   \\
& \hphantom{XXXXXXX}
\times 
e^{i  \Xi_\NLS (n_{123}^*, n_1, n_2, n_3)w(r)}
\prod_{j = 1}^5{\ft f_j}^*(n_j)\\
& \quad  + 
   \sum_{\substack{n = n_{12345}^*\\
n \ne n_1, n_5\\
n_{234}^* \ne n_2, n_4
}}
 \Phi^w_{t, r}(\Xi_\NLS (n, n_1, n_{234}^*,  n_5))
  \\
& \hphantom{XXXXXXX}
\times 
e^{- i  \Xi_\NLS (n_{234}^*, n_2, n_3, n_4)w(r)}
\prod_{j = 1}^5{\ft f_j}^*(n_j)\\
& =: \F_x\big(\1_{t, r}(f_1, \dots, f_5)\big)(n) + \F_x\big(\II_{t, r}(f_1, \dots, f_5)\big)(n)
\end{split} 
\label{xX5}
\end{align}

\noi
for $0 \le r< t  \le T$, 
where we used the conventions \eqref{no2}
and \eqref{gauge2}.
The second operator
$\NN \RR^\NLS_{t, r}(f_1, \dots, f_5)$
 represents the contribution
coming from the resonant part in substituting~\eqref{xNLS2a}
and is defined by 
\begin{align}
\begin{split}
  \F_x  & \big(\NN \RR^\NLS_{t, r}(f_1, \dots, f_5)\big)(n)\\
&  =    
- 2   
\sum_{\substack{n = n_1 - n_4 + n_5\\ n\neq n_{1},n_{5}}}
 \Phi^w_{t, r}(\Xi_\NLS (n, n_1, n_4, n_5))
 \\
& \hphantom{XXXXXXXX}
\times 
\F_x\big(\RR(f_1, f_2, f_3)\big)(n_1)
 \prod_{j = 4}^5{\ft f}_j^*(n_j)\\
& \quad 
+ 
\sum_{\substack{n = n_1 - n_2 + n_5\\ n\neq n_{1},n_{5}}}
 \Phi^w_{t, r}(\Xi_\NLS (n, n_1, n_2, n_5))
  \\
& \hphantom{XXXXXXXX}
\times 
\F_x\big(\RR(f_2, f_3, f_4)\big)(n_2)
\cdot \prod_{j \in \{1, 5\}} {\ft f}_j^*(n_j), 
\end{split} 
\label{xX6}
\end{align}

\noi
where $\RR$ is as in \eqref{non1}.

Hence, putting 
\eqref{xNLS3}, \eqref{xNLS4}, \eqref{xX5}, and \eqref{xX6}
together, we obtain
 the following normal form equation for 
the 
 modulated cubic NLS \eqref{xNLS2}:
\begin{align}
\begin{split}
\ft \uu(t, n) - \ft \uu(0, n) 
& =    \F_x\big(X^\NLS_{t, 0}(\uu(0))\big)(n)  
   -i   \int_0^t \F_x\big(\RR  (\uu(t'))\big) (n)dt'\\
& \quad  +  \int_0^t 
 \F_x\big(\NN\NN^\NLS_{t, t'}(\uu(t'))\big)(n)dt'\\
& \quad   +  \int_0^t 
 \F_x\big(\NN\RR^\NLS_{t, t'}(\uu(t'))\big)(n)dt', 
\end{split}
\label{xNLS5}
\end{align}

\noi
where we used the short-hand notation \eqref{short1}.

We now state lemmas on boundedness properties
of the multilinear operators appearing in~\eqref{xNLS5}.
Once we prove these lemmas
(Lemmas \ref{LEM:nonlin7}, 
\ref{LEM:nonlin8}, and \ref{LEM:nonlin9}), 
Theorem~\ref{THM:UU4} 
follows
from proceeding as in 
the proof of Theorem~\ref{THM:UU1}
for the modulated KdV~\eqref{kdv1}
presented in Section \ref{SEC:KDV}
with the observation that,   
given a solution  $u$
to~\eqref{xNLS1}
in the class
\[   C([0, \tau]; H^\frac 16(\T))
\subset C([0, \tau]; L^3(\T))
\]

\noi
for some $0 < \tau \le T$, 
 $\ft \uu (\cdot, n)$
is a $C^1$-function in $t$
for each fixed  $n \in \Z$, 
which is a key point in justifying the formal computations
in deriving the normal form equation \eqref{xNLS5};
see also~\cite{GKO, OW}.
Since this is standard, we omit details of the proof of 
Theorem~\ref{THM:UU4}.

\medskip

The boundedness of $\RR$ in \eqref{non1} follows
trivially.

\begin{lemma}\label{LEM:nonlin7}
Let $\RR$ be as in \eqref{non1}.
Then, given  any 
 $ s \ge 0$, 
we have 
\begin{align}
 \|\RR\|_{\L_3(H^s; H^{s})}
\les 1.
\label{non7}
\end{align}

\end{lemma}

\begin{proof}
By Plancherel's identity, we see that the contribution from the first two terms on the right-hand side of \eqref{non1}
satisfies the bound \eqref{non7}, 
as long as $s \ge 0$.
As for the 
last term on the right-hand side of \eqref{non1}, 
the bound \eqref{non7} follows from $\jb{n}^s \le \jb{n}^{3s }$
for $s \ge 0$
and $\| a_n \|_{\l^6_n}\le \| a_n \|_{\l^2_n}$.
\end{proof}

Next, we establish boundedness of 
the trilinear operator  $X^\NLS_{t, r}$
defined in  \eqref{xX1}.

\begin{lemma}\label{LEM:nonlin8}
Given $\rho > \frac 12 $,  $\frac12< \g < 1$, and $T> 0$, 
let  $w$ be $(\rho,\g)$-irregular on $[0, T]$ in the sense of Definition~\ref{DEF:ir}.
Given $0 \le r < t \le T$, let $X^\NLS_{t, r}$
be as in \eqref{xX1}.
Then, 
given any $s \ge 0$, 
we have $X_{t, r}^\NLS \in \L_3(H^s(\T); H^{s}(\T))$
with the following bound\textup{:}
\begin{align*}
 \|X^\NLS_{t, r}\|_{\L_3(H^s; H^{s})}
\les 
(t-r)^\g  \|\Phi^w\|_{  \W^{\rho,\g}_T}
\end{align*}

\noi
for any $0 \le r < t \le T$, 
where the implicit constant is independent of $T > 0$.

\end{lemma}

\begin{proof}

From 
\eqref{xX2},  Cauchy-Schwarz's inequality, 
\eqref{rho1}, and \eqref{Xi9} (note that $n \ne n_1, n_3$)
with the triangle inequality
 $\jb{n_{123}^*}^s \les \jb{n_1}^s\jb{n_2}^s\jb{n_3}^s$
for $s \ge 0$, 
we have 
\begin{align*}
\| X^\NLS_{t,r}  ( f _1, f _2, f_3) \|_{H^s}
& \les 
(t - r)^\g
\|\Phi^w\|_{  \W^{\rho,\g}_T} 
\Bigg\|\bigg( \sum_{n_1, n_3 \in \Z}
\frac {1}{\jb{n - n_1}^{2\rho}\jb{n - n_3}^{2\rho}}
\bigg)^\frac 12\\
& \quad \times \bigg(  
\sum_{n = n_{123}^*}
\prod_{j = 1}^3
\jb{n_j}^{2s}
 |\ft  f_j ( n_j )|^2\bigg)^\frac 12 
 \Bigg\|_{\l^2_n}\\
 & \les
 \|\Phi^w\|_{  \W^{\rho,\g}_T} 
(t - r)^\g
\prod_{j = 1}^3 \|f_j\|_{H^s}, 
\end{align*}

\noi
provided that $\rho > \frac 12$.
\end{proof}

Lastly,  we establish boundedness of 
the quintilinear operators  
$\NN \NN^\NLS_{t, r}(f_1, \dots, f_5)$
and 
$\NN \RR^\NLS_{t, r}(f_1, \dots, f_5)$
defined in  \eqref{xX5}
and \eqref{xX6}.

\begin{lemma}\label{LEM:nonlin9}
\textup{(i)}
Given $\rho > \frac 12 $,  $\frac12< \g < 1$, and $T> 0$, 
let  $w$ be $(\rho,\g)$-irregular on $[0, T]$ in the sense of Definition~\ref{DEF:ir}.
Given $0 \le r < t \le T$, let $\NN \RR^\NLS_{t, r}$
be as in \eqref{xX6}.
Then, given  $s \ge 0$, 
$\NN \RR^\NLS_{t, r}$  belongs to $\L_5(H^s(\T); H^{s}(\T))$, 
satisfying
 the following bound\textup{:}
\begin{align}
 \|\NN \RR^\NLS_{t, r}\|_{\L_5(H^s; H^{s})}
\les 
(t- r)^\g  \|\Phi^w\|_{  \W^{\rho,\g}_T}.
\label{Xnox1c}
\end{align}

\medskip

\noi
\textup{(ii)}
Given $\rho > \frac 23 $,  $\frac12< \g < 1$, and $T> 0$, 
let  $w$ be $(\rho,\g)$-irregular on $[0, T]$ in the sense of Definition~\ref{DEF:ir}.
Given $0 \le r < t \le T$, let $\NN \NN^\NLS_{t, r}$
be as in \eqref{xX5}.
Then, given  $s \ge \frac 16$, 
$\NN \NN^\NLS_{t, r}$  belongs to $\L_5(H^s(\T); H^{s}(\T))$, 
satisfying
 the following bound\textup{:}
\begin{align}
 \|\NN \NN^\NLS_{t, r}\|_{\L_5(H^s; H^{s})}
\les 
(t- r)^\g  \|\Phi^w\|_{  \W^{\rho,\g}_T}.
\label{Xnox1d}
\end{align}

\end{lemma}

\begin{proof}
(i) 
From 
\eqref{xX6} and  \eqref{xX2},  we have 
\begin{align*}
\NN \RR^\NLS_{t, r}(f_1, \dots, f_5)
& = -2i X^\NLS_{t,r}  ( \RR(f_1, f_2, f_3) , f _4, f_5)\\
& \,\quad + i X^\NLS_{t,r}  ( f_1, \RR(f_2, f_3, f_4) ,  f_5).
\end{align*}

\noi
Hence, the bound
\eqref{Xnox1c} follows from Lemmas \ref{LEM:nonlin7}
and \ref{LEM:nonlin8}, 
provided that $\rho > \frac 12 $ and $s \ge 0$.

\medskip

\noi
(ii) 
In view of the triangle inequality $\jb{n_{1\cdots5}^*}^\s \les \prod_{j= 1}^5 \jb{n_j}^\s$
for $\s \ge 0$, we only consider the case $s = \frac 16 $.

We first estimate $\1_{t, r}(f_1, \dots, f_5)$ in \eqref{xX5}.
\medskip

\noi
$\bul$ {\bf Case (a):} $|n|\gg |n_5|$.\\ 
\indent
From \eqref{xX5}, \eqref{rho1}, and \eqref{Xi9}
with 
$\jb{n} \sim \jb{n - n_5}$, 
we have 
\begin{align}
\begin{split}
 \|\1_{t, r}(f_1, \dots, f_5)\|_{H^\frac 16}
&  
\les 
(t - r)^\g
\|\Phi^w\|_{  \W^{\rho,\g}_T} 
\bigg\| \sum_{m = n_{123}^*}
\prod_{j = 1}^3
{\ft f_j}^*(n_j)\bigg\|_{\l^\infty_{m}}\\
& \quad \times \bigg\| 
\sum_{n = m - n_4 + n_5}
\frac {|\ft f_4(n_4)|  |\ft f_5(n_5)|}{\jb{n - m}^{\rho}\jb{n - n_5}^{\rho - \frac 16}}
\bigg\|_{\l^2_n}.
\end{split}
\label{non91}
\end{align}

\noi
By Hausdorff-Young's and Sobolev's inequalities, we have 
\begin{align}
\bigg\| \sum_{m = n_{123}^*}
\prod_{j = 1}^3
{\ft f_j}^*(n_j)\bigg\|_{\l^\infty_{m}}
\le \prod_{j = 1}^3 \|f_j\|_{L^3}\les \prod_{j = 1}^3 \|f_j\|_{H^\frac 16}.
\label{non92}
\end{align}

\noi
On the other hand, 
using  $\jb{n} \sim \jb{n - n_5}$
and applying Young's,  
Hausdorff-Young's,  and Sobolev's inequalities, we have 
the last factor on the right-hand side of \eqref{non91} is bounded by 
\begin{align}
\begin{split}
& \les 
\bigg\|\frac 1{\jb{n}^{\rho - \frac 16}}
\bigg\|_{\l^2_n}
\cdot \sum_{n_4, n_5 \in \Z}
\frac {|\ft f_4(n_4)|  |\ft f_5(n_5)|}{\jb{n_4 - n_5}^{\rho}}\\
& \les 
\| \ft f_4(n_4)\|_{\l^\frac 32_{n_4}}  
\bigg\|
\sum_{n_5 \in \Z}
\frac{\ft f_5(n_5)}{\jb{n_4 - n_5}^{\rho}}\bigg\|_{\l^3_{n_4}}\\
& \les \prod_{j = 4}^5 \|f_j\|_{L^3}\les \prod_{j = 4}^5 \|f_j\|_{H^\frac 16}, 
\end{split}
\label{non93}
\end{align}

\noi
provided that $\rho > \frac 23$.
Hence, the bound \eqref{Xnox1d}
follows from \eqref{non91}, 
\eqref{non92}, and 
\eqref{non93} in this case.

\medskip

\noi
$\bul$ {\bf Case (b):} $|n|\les |n_5|$.\\ 
\indent
By $\jb{n}^\frac 16 \les\jb{n_5}^\frac 16$, we have 
\begin{align}
\begin{split}
 \|\1_{t, r}(f_1, \dots, f_5)\|_{H^\frac 16}
&  
\les 
(t - r)^\g
\|\Phi^w\|_{  \W^{\rho,\g}_T} 
\bigg\| \sum_{m = n_{123}^*}
\prod_{j = 1}^3
{\ft f_j}^*(n_j)\bigg\|_{\l^\infty_{m}}\\
& \quad \times \bigg\| 
\sum_{n = m - n_4 + n_5}
\frac {|\ft f_4(n_4)|  \jb{n_5}^\frac 16 |\ft f_5(n_5)|}{\jb{n - m}^{\rho}\jb{n - n_5}^{\rho}}
\bigg\|_{\l^2_n}, 
\end{split}
\label{non94}
\end{align}

\noi
where the last factor is bounded by 
\begin{align}
\begin{split}
& \les 
\bigg\|
\sum_{n_4 \in \Z}
\frac{\ft f_4(n_4)}{\jb{n_4 - n_5}^{\rho}}\bigg\|_{\l^3_{n_5}}
\Bigg\|
\bigg\|
\frac{  \jb{n_5}^\frac 16 |\ft f_5(n_5)|}{\jb{n - n_5}^{\rho}}\bigg\|_{\l^\frac 32_{n_5}}
\Bigg\|_{\l^2_n}\\
& \les  \|f_4\|_{L^3}
\|f_5\|_{H^\frac 16}
\les \prod_{j = 4}^5 \|f_j\|_{H^\frac 16}, 
\end{split}
\label{non95}
\end{align}

\noi
provided that $\rho > \frac 23$.
Hence, the bound \eqref{Xnox1d}
follows from \eqref{non94}, 
\eqref{non92}, and 
\eqref{non95} in this case.

\medskip

Next, we estimate
the second term
 $\II_{t, r}(f_1, \dots, f_5)$ in \eqref{xX5}.

\medskip

\noi
$\bul$ {\bf Case (c):} $|n|\gg \max(|n_1|, |n_5|)$.\\ 
\indent
Proceeding as in \eqref{non91} with $\jb{n- n_5} \sim \jb{n}$, we have
\begin{align}
\begin{split}
 \|\II_{t, r}(f_1, \dots, f_5)\|_{H^\frac 16}
&  
\les 
(t - r)^\g
\|\Phi^w\|_{  \W^{\rho,\g}_T} 
\bigg\| \sum_{m = n_{234}^*}
\prod_{j = 2}^4
{\ft f_j}^*(n_j)\bigg\|_{\l^\infty_{m}}\\
& \quad \times \bigg\| 
\sum_{n = n_1 -m + n_5}
\frac {|\ft f_1(n_1)|  |\ft f_5(n_5)|}{\jb{n - n_1}^{\rho}\jb{n}^{\rho - \frac 16}}
\bigg\|_{\l^2_n}.
\end{split}
\label{non96}
\end{align}

\noi
With $\jb{n-n_1}\ges \jb{n_1 - n_5}$, we can bound the last factor 
of the right-hand side of \eqref{non96}
by~\eqref{non93} (after relabelling), 
while the penultimate factor of the right-hand side of \eqref{non96}
can be bounded by~\eqref{non92} (after relabelling).
This proves \eqref{Xnox1d} in this case, provided that $\rho > \frac 23$.

\medskip

\noi
$\bul$ {\bf Case (d):} $|n|\les \max(|n_1|, |n_5|)$.\\ 
\indent
Without loss of generality, assume that $|n_5|\ges |n_1|$.
Then, 
proceeding as in \eqref{non94},  we have 
\begin{align}
\begin{split}
 \|\II_{t, r}(f_1, \dots, f_5)\|_{H^\frac 16}
&  
\les 
(t - r)^\g
\|\Phi^w\|_{  \W^{\rho,\g}_T} 
\bigg\| \sum_{m = n_{234}^*}
\prod_{j = 2}^4
{\ft f_j}^*(n_j)\bigg\|_{\l^\infty_{m}}\\
& \quad \times \bigg\| 
\sum_{n = n_1-m  + n_5}
\frac {|\ft f_1(n_1)|  \jb{n_5}^\frac 16 |\ft f_5(n_5)|}{\jb{n - n_1}^{\rho}\jb{n - n_5}^{\rho}}
\bigg\|_{\l^2_n}, 
\end{split}
\label{non97}
\end{align}

\noi
where the last factor is bounded by 
\begin{align}
\begin{split}
& \les 
\bigg\|
\sum_{n_1 \in \Z}
\frac{\ft f_1(n_1)}{\jb{n - n_1}^{\rho}}\bigg\|_{\l^3_{n}}
\bigg\|
\sum_{n_5 \in \Z} \frac{  \jb{n_5}^\frac 16 |\ft f_5(n_5)|}{\jb{n - n_5}^{\rho}}\bigg\|_{\l^6_{n}}\\
& \les  \|f_1\|_{L^3}
\|f_5\|_{H^\frac 16}
\les \prod_{j \in \{1, 5\}} \|f_j\|_{H^\frac 16}, 
\end{split}
\label{non98}
\end{align}

\noi
provided that $\rho > \frac 23$.
Hence, the bound \eqref{Xnox1d}
follows from \eqref{non97}, 
and \eqref{non98}
with 
\eqref{non92} (after relabelling).
\end{proof}

\begin{remark} \label{REM:NLS1}\rm

(i)
By slightly modifying the normal form reduction, 
 it is possible to prove nonlinear smoothing 
for \eqref{xNLS1},
as in Theorems \ref{THM:UU1}, \ref{THM:UU2}, and \ref{THM:UU3}.
More precisely, 
as  in 
\cite{GKO, OW}, 
we can implement a normal form reduction
by introducing a suitable cutoff $K$ for the resonance function $\Xi_{\NLS}(\bar n)$ in \eqref{Xi9}
and thus separately considering
the {\it nearly resonant} contribution 
(coming from $|\Xi_{\NLS}(\bar n)|\ll K$)
and the {\it highly
non-resonant} contribution
(coming from $|\Xi_{\NLS}(\bar n)|\ges K$).
For example, take $K \sim n_{\max}$.
In  the nearly resonant case: $|\Xi_\NLS(\bar n)| \ll n_{\max}$, 
where 
$\Xi_\NLS(\bar n)$ and 
$n_{\max}$ are as in 
\eqref{Xi9} and \eqref{max1}, respectively, 
we have
\begin{align*}
 |n|\sim |n_1|\sim|n_2|\sim|n_3|,
\end{align*}

\noi
which implies nonlinear smoothing, 
provided that $s_0 \le 3s$
(possibly with further conditions).
On the other hand, 
it follows from~\eqref{rho1} 
and $|\Xi_\NLS(\bar n)| \ges  n_{\max}$
that  nonlinear smoothing of arbitrary degree
holds 
for the highly non-resonant part
provided that $\rho$ is sufficiently large.
For conciseness of the presentation, we omit details.

\medskip

\noi
(ii) 
We note that splitting the nonlinearity
into the nearly resonant and highly non-resonant parts, 
as explained in Part (i), 
is the basis of (a practical application) the normal form method
as seen in \cite{GKO, OW}.
In the modulated setting, 
it follows from  the identity~\eqref{NS1} with 
\eqref{rho1}
that 
the  highly
non-resonant contribution
(coming from $|\Xi(\bar n)|\ges K$, 
where $\Xi(\bar n)$ denotes the resonance function of a given equation)
enjoys a gain of $K^{-\rho}$, 
while the nearly resonant contribution
benefits from the frequency restriction coming from 
$|\Xi(\bar n)|\ll K$.
This viewpoint applies to a general modulated dispersive PDE, not restricted to one-dimensional models.
For example, 
there are unconditional uniqueness results 
\cite{Kishi21, BOW} 
for higher dimensional (unmodulated dispersive) PDEs
 via the normal form method.
We expect that 
our normal form approach 
with the identity~\eqref{NS1}
applies to 
the modulated versions of these 
higher dimensional models.

\end{remark}

\begin{remark}\label{REM:NLS2}\rm
(i)
In the unmodulated setting, 
an infinite iteration of normal form reductions
was needed to prove unconditional uniqueness
of the (unmodulated) cubic NLS \eqref{NLS3} in $H^\frac 16(\T)$;
see \cite{GKO}.
On the other hand, we only needed to 
implement one normal form reduction
to prove Theorem \ref{THM:UU4},
which can be seen as an instance of regularization by noise.

\medskip

\noi
(ii)
The regularity restriction $s \ge \frac 16$
appears in 
\eqref{non92}, 
which shows that if we were to stop normal form reductions
within a finite number of steps, 
we would not be able to go below 
the regularity threshold $ s= \frac 16$.
Hence, as in  the case of the unmodulated cubic NLS  \eqref{NLS3}
studied in 
 \cite{GKO, OW}, 
we expect that an infinite iteration of normal form reductions is needed
to go below $ s= \frac 16$.
We, however,  point out
that this comes with a twist due to 
the fact that the identity \eqref{NS0}
is replaced by 
\eqref{NS1} in the modulated setting.
We will address this issue in a forthcoming paper.

\end{remark}

\section{Declarations}

\noi
{\bf Funding.}
M.G.~was supported by 
the UKRI Frontier Research Grant 
(grant no.~EP/Z534328/1 ``Stochastic analysis of quantum fields"). 
G.L. was supported by the NSFC (grant no.~12501181).
T.O.~was supported by the European Research Council (grant no.~864138 ``SingStochDispDyn")
and  by the EPSRC 
Mathematical Sciences
Small Grant  (grant no.~EP/Y033507/1).
T.O. also 
acknowledges support from  
the NSFC (grant no.~W2531005).

\medskip

\noi
{\bf Competing interests.}
The authors have no competing interests to declare that are relevant to the content of this article.

\begin{ackno}\rm
The authors would like to thank Andreia Chapouto
for a comment on numerical schemes
in Remark \ref{REM:num}.
They would also  like to express their gratitude to the anonymous
referees for the helpful comments which improved the quality of the paper.

\end{ackno}

\end{document}